\documentclass[12pt]{article}
\usepackage{amscd, amsmath}
\usepackage{amsthm, amssymb}

\usepackage[hang,flushmargin]{footmisc}

\newtheorem{theorem}{Theorem}[section]
\newtheorem{lemma}[theorem]{Lemma}
\newtheorem{proposition}[theorem]{Proposition}

\newtheorem*{assumption(A)}{Condition (A)}
\newtheorem{conjecture}[theorem]{Conjecture}

\theoremstyle{remark}

\theoremstyle{remark}
\newtheorem*{remark}{{\bf Remark}}
\newtheorem*{example}{{\bf Example}}

\numberwithin{equation}{section}

\newcommand{\CC}{\mathbb{C}}
\newcommand{\RR}{\mathbb{R}}
\newcommand{\dd}{{\rm d}}
\newcommand{\MM}{\mathcal M}
\newcommand{\A}{\mathcal A}
\newcommand{\LL}{\mathcal L}

\makeatletter
\def\blfootnote{\xdef\@thefnmark{}\@footnotetext}
\makeatother

\begin{document}

\begin{titlepage}
\title{\bf Moduli Spaces of Abelian Vortices on K\"ahler Manifolds}
\vskip -70pt
\vspace{3cm}

\author{{J. M. Baptista}}  

\date{August 2013}

\maketitle

\thispagestyle{empty}
\vspace{2cm}
\vskip 20pt
{\centerline{{\large \bf{Abstract}}}}
\noindent
We consider the self-dual vortex equations on a positive line bundle $L \rightarrow M$ over a compact K\"ahler manifold of arbitrary dimension. When $M$ is simply connected, the moduli space $\MM$ of vortex solutions is a projective space. When $M$ is an abelian variety, $\MM$ is the projectivization of the Fourier-Mukai transform $\hat{L} \rightarrow \hat{M}$ over the dual variety. We extend this description of the moduli space to the abelian GLSM, i.e. to vortex equations with a torus gauge group acting linearly on a complex vector space. After establishing the Hitchin-Kobayashi correspondence appropriate  for the general abelian GLSM, we give explicit descriptions of the moduli space $\MM$ in the case where the manifold $M$ is simply connected or is an abelian variety. In these examples we compute the K\"ahler class of the natural $L^2$-metric on the moduli space. In the simplest cases we compute  the volume and total scalar curvature of $\MM$. Finally, we note that for abelian GLSM the moduli space $\MM$ is a compactification of the space of holomorphic maps from $M$ to toric targets, just as in the usual case of $M$ being a Riemann surface. This leads to various natural conjectures, for instance explicit formulae for the volume of the space of maps $ \CC \mathbb{P}^m \rightarrow \CC \mathbb{P}^n$.


\end{titlepage}

\tableofcontents

\section{Introduction}

The simplest type of vortex equations on a K\"ahler manifold is defined  on a hermitian line bundle $L \longrightarrow M$ over the manifold. The variables are a unitary connection $A$ and a smooth section $\phi$ of $L$. To each pair $(A, \phi)$ one associates a positive energy
\begin{equation}\label{energy_functional}
E(A, \phi) \ = \ \int_M \frac{1}{2\,e^2}\, \vert F_A \vert^2 \ + \ \vert \dd_A \phi \vert^2  \ + \ \frac{e^2}{2} \, \big( \vert \phi \vert^2 - \tau \big)^2 \ , 
\end{equation}
where $\tau$ and $e^2$ are positive constants, $\dd_A \phi$ is the covariant derivative, and $F_A$ denotes the curvature form of the connection $A$. All the norms are taken with respect to the hermitian metric on $L$ and the K\"ahler metric on $M$. In two dimensions, this energy functional coincides with the usual static energy of the abelian Higgs model.
Assuming that $M$ is compact, a standard Bogomolny-type argument \cite{Brad} shows that the energy is bounded below by a constant,
\begin{equation}
E(A, \phi) \ \geq \ E_{vortex}  \ , 
\end{equation}
and that this minimum is attained precisely when the pair $(A, \phi)$ satisfies the so-called self-dual vortex equations. The equations read
\begin{align}
&\bar{\partial}_A \phi \ = \ 0   \label{vortex1} \\
&\Lambda F_A  \: -\:  i  e^2\:  \big( |\phi|^2 - \tau \big) \ = \ 0 \ \label{vortex2} \\
& F_A^{0,2} \ = \ 0 \ , \label{vortex3}
\end{align}
where $\Lambda F_A$ represents the contraction of the curvature with the K\"ahler form $\omega_M$ on $M$. When $M$ is a Riemann surface, $\Lambda F_A = \ast F_A$ is just the Hodge-dual and the last equation is trivially satisfied, since every 2-form has vanishing (0,2) Dolbeault component. The energy of a vortex solution on a $m$-dimensional compact K\"ahler manifold is given by
\begin{equation}\label{energy_vortex}
E_{vortex} \ = \  \int_M \frac{i \tau}{(m-1)!}\, F_A \wedge \omega_M^{m-1} \ + \ \frac{1}{2 e^2 \, (m-2)!} \, F_A \wedge F_A \wedge \omega_M^{m-2} \ .
\end{equation}
It does not depend on the particular solution $(A, \phi)$ because the cohomology class $[F_A]$ is just the Chern class $-2\pi i \, c_1(L)$, a topological invariant of the line bundle $L$.

In this article we will study the moduli space of solutions these vortex equations and of another, more general, flavour of vortex equations: those of the abelian gauged linear sigma-model (GLSM). These are written down in \eqref{v-eq}. All these equations and moduli spaces have been widely studied in the case where $M$ is a Riemann surface \cite{JT, Brad, Garcia-Prada, Morrison-Plesser, MS, Wehrheim}. 
In fact, in two dimensions, the subject has advanced much further than the abelian case, with an impressive body of work about non-abelian vortex equations, i.e. equations defined on vector bundles over $M$ (see for instance \cite{Brad2, Thaddeus, BDGW, EINOS, Tong, Bap2009} and the references therein), and vortex equations on bundles with non-linear fibres (\cite{CGMS, Woodward, Ziltener} and the references therein). 
When $M$ is a higher-dimensional K\"ahler manifold, however,  our understanding of vortex moduli spaces is much less thorough. Important results have already been obtained, such as a version of the Hitchin-Kobayashi correspondence \cite{Brad, Banfield, MiR}. In the abelian case, the simplest vortex equations were related to invariant Hermitian-Einstein connections and their moduli spaces of solutions were described as spaces of divisors on $M$ \cite{Garcia-Prada, Bap2006}. Nonetheless, for higher-dimensional $M$ there seems to be a general lack of explicit descriptions of the moduli spaces as global manifolds, even in the case of the simplest abelian vortices. Here we try to partially fill in this gap.

In Section 2 of this article we go through the basic analysis for the simplest abelian vortices, i.e. for equations \eqref{vortex1}-\eqref{vortex3} defined on a line bundle $L$. In this case, the main result of \cite{Brad} implies that the vortex moduli space $\MM$ is a projective space when $M$ is simply connected. When $M$ is an abelian variety, the moduli space is a projective bundle over dual variety $\hat{M}$;  more precisely, $\MM$ is the projectivization of the Fourier-Mukai transform $\hat{L} \rightarrow \hat{M}$ of the original line bundle $L$. All the results here are either well known, or are relatively simple calculations using standard technology. Nevertheless, we feel that this Section gives a useful overview of the techniques developed ahead, besides being important to introduce the necessary notation about vortex moduli spaces, abelian varieties and the universal bundle.

In Section 3 we consider vortex solutions of the general abelian GLSM. The main results are Theorems \ref{Hitchin_Kobayashi} and \ref{tau-stability}, which describe the Hitchin-Kobayashi correspondence for the GLSM and interpret the general $\tau$-stability of \cite{Banfield, MiR} in this setting. These results are applied in Sections 3.2 and 3.4, where we show that the vortex moduli space of the GLSM is a toric orbifold for simply connected $M$,  and that, for abelian varieties, it is a toric fibration over a cartesian product of copies of the dual variety. We also work out the topology of the natural ``universal" bundle over the product $M \times \MM$, a result necessary for the cohomology calculations in Section 4.

General arguments say that vortex moduli spaces have a natural K\"ahler structure \cite{MiR}. The complex structure on $\MM$ is canonically induced by the complex structures on $M$ and on $L$. The compatible K\"ahler metric $\omega_\MM$ can be described roughly as follows. If $(A_t, \phi_t)$ is a one-parameter family of vortex solutions, a tangent vector to this path is represented by the derivative $\dot{\phi}$ -- a section of $L$ -- and the derivative $\dot{A}$ -- a 1-form on $M$. These derivatives satisfy the linearized vortex equations. The standard metric on the moduli space $\MM$ is then determined by the $L^2$-norm
\begin{equation} \label{vortexmetric}
\Vert \dot{A} + \dot{\phi} \Vert^2_{\omega_\MM} \ = \ \int_M \Big(\, \frac{1}{4e^2} \: \vert \dot{A} \vert^2 \ + \ \vert \dot{\phi} \vert^2 \, \Big)\, \frac{\omega_M^m}{m!}
\end{equation}
applied to the horizontal part of the tangent vector, i.e., to the component of $( \dot{A} , \dot{\phi})$ perpendicular to all vectors tangent to gauge transformations. This metric on  $\MM$ has a well known interpretation as the reduced metric on a (infinite-dimensional) symplectic quotient. In physics, it gives relevant information about the low-energy dynamics of vortices \cite{MS}. It is very hard in general to obtain explicit knowledge about $\omega_\MM$. Our work in Section 4 is to compute the cohomology class $[\omega_\MM]$ of the K\"ahler form on the moduli space. We are able to do so both for the simpler model of vortices on line bundles and for the more general abelian GLSM. 
The base $M$, as always, can either be simply connected or be an abelian variety. In the simplest examples, our knowledge of the K\"ahler class $[\omega_\MM]$ and of the cohomology ring of $\MM$ allows us to write down explicit formulae for the volume and total scalar curvature of $(\MM , \omega_\MM)$.  The results of this Section extend formulae obtained by  Manton and Nasir \cite{MN} and the author  \cite{Bap2010} in the case where $M$ is a Riemann surface. We rely heavily on a description of $\omega_\MM$ originally due to Perutz \cite{Perutz} and generalized in  \cite{BS, Bap2010}.

In the final part of the paper, Section 5, we look at the relation between vortex moduli spaces in abelian GLSM and spaces of holomorphic maps $M \rightarrow X$ to toric targets. This is a very familiar and well studied story in the case where $M$ is a Riemann surface. In fact, in two dimensions, one of the main motivations to consider abelian GLSM is that  its vortex moduli space is a natural compactification of the space of holomorphic curves in toric varieties, and hence can be used to compute  Gromov-Witten invariants of these varieties \cite{Witten, Morrison-Plesser, Mirror Symmetry}. 
Much less attention has been devoted to this line of ideas when $M$ is a higher-dimensional K\"ahler manifold. Since many of the techniques extend readily from the two dimensional case, in Section 5 we go through this natural generalization, describing the natural embedding $\mathcal{H} \hookrightarrow \MM$ of the (generally non-compact) space of holomorphic maps into the corresponding (compact) vortex moduli space (Proposition \ref{embedding}). One notable difference from the two dimensional case is that $\mathcal{H}$ does not generally embed as an open and dense subset of $\MM$. Roughly speaking, this will happen only if the dimension of $M$ is small enough compared to the dimension of the target $X$ (Proposition \ref{open_embedding_condition}) and the line bundles of the GLSM have enough holomorphic sections. 

One of our original motivations to study the embedding $\mathcal{H} \hookrightarrow \MM$ is to compare the natural K\"ahler metrics on both these spaces. Recall that $\mathcal{H}$ is a space of maps between Riemannian manifolds, and hence has a natural $L^2$-metric $\omega_\mathcal{H}$. It is natural to ask if there is any relation between $\omega_\mathcal{H}$ and the pullback of the vortex metric $\omega_\MM$ by the embedding $\mathcal{H} \hookrightarrow \MM$. Following \cite{Bap2010}, in Section 5.3 we conjecture that the metric  $\omega_\mathcal{H}$ is the pointwise limit of $\omega_\MM$ when the gauge coupling constant $e^2$ goes to infinity.\footnote{This conjecture has now been proved in the case of maps from Riemann surfaces to projective space. The proof is presented in \cite{Liu}, and appeared on the arXiv after the first version of the present article.} In particular, when $\mathcal{H}$ embeds as an open dense subset of $\MM$, we should have that
\[
\text{\rm Vol}\, (\mathcal{H}, \omega_{\mathcal{H}}) \ = \ \lim_{e^2 \rightarrow + \infty} \: \text{\rm Vol}\, (\MM, \omega_{\MM })  \ .
\]
Taking the limit of the explicit formulae for $\text{\rm Vol}\, (\MM, \omega_{\MM })$ obtained in Section 4, we can thus propose a few conjectural formulae for $\text{\rm Vol}\, (\mathcal{H}, \omega_{\mathcal{H}})$, as well as for the total scalar curvature of $(\mathcal{H}, \omega_{\mathcal{H}})$. See for instance the Example in Section 5.3. It seems non-trivial to check these formulae by direct computations on $\mathcal{H}$, since this space is not compact and the metric $\omega_\mathcal{H}$ cannot in general be compactified as a smooth metric (it can have unbounded scalar curvature). In the very special instances where direct computations have been performed \cite{Speight}, they agreed with analogous conjectures made in \cite{Bap2010}.

$\ $


\section{Vortices on line bundles}

\subsection{General description}

The simplest vortex equations on a $m$-dimensional K\"ahler manifold are defined  on a hermitian line bundle $L \rightarrow M$. They were written down in \eqref{vortex1}-\eqref{vortex3}. 
From equation \eqref{vortex3} we see that a necessary condition for the existence of vortex solutions is that
\begin{equation}\label{type_chern_class}
c_1(L) \ \in \ H^{1,1}(M, \mathbb{C}) \ ,
\end{equation}
i.e. that $L$ admits a holomorphic structure. Assuming that $M$ is compact, the integral of the second equation over $(M, \omega_M)$ shows that another necessary condition for the existence of solutions with non-trivial section $\phi$ is that
\begin{equation}\label{stability1}
\sigma \ := \  \tau \, \text{\rm Vol}\, M  \  - \  \frac{2 \pi \, m}{e^2} \, \, \frac{c_1(L)^\parallel}{[\omega_M]}\, \text{\rm Vol}\, M \ > \ 0 \ ,
\end{equation}
where $\parallel$ denotes the component parallel to the vector $[\omega_M]$ in the euclidean space $H^2 (M, \mathbb{R})$. See Section 3.1 for more details.
The real number $\sigma$ will be called the stability parameter. When these two conditions are satisfied, we can have plenty of vortex solutions, as follows from the basic result of Bradlow:
\begin{proposition}\cite{Brad}
Assume that condition \eqref{stability1} is satisfied. Then for any pair 
$(A, \phi)$ with non-zero $\phi$ satisfying $\bar{\partial}_A \phi = 0$ and $F_A^{0,2} = 0$, there exists a gauge transformation $g : M
\rightarrow \CC^\ast $, unique up to multiplication by  $U(1)$-gauge
transformations, such that $(g(A), g(\phi))$ is a solution of
the full vortex equations. All vortex solutions can be obtained in this way.
\end{proposition}
\noindent
So the set of full vortex solutions up to $U(1)$-gauge transformations coincides with the set of solutions of \eqref{vortex1} and \eqref{vortex3} up to $\CC^\ast$-gauge transformations. This is an example of a Hitchin-Kobayashi correspondence. As is well-known,  there is also a bijective correspondence between connections on $L\rightarrow M$ satisfying $F_A^{0,2} =0$ and holomorphic structures on the line bundle \cite{Kobayashi}. Moreover, if $A$ is such a connection, then $\bar{\partial}_A \phi \ = \ 0$ if and only if the section $\phi$ is holomorphic with respect to the holomorphic structure on $L$ induced by $A$. It follows that the moduli space $\MM$ of vortex solutions coincides the space of pairs ``holomorphic structure plus a non-zero holomorphic section" on $L$, up to isomorphism. It is also a standard fact that such pairs are in one-to-one correspondence with effective divisors on $M$ representing the cohomology class $c_1(L)$. So we have:

\begin{proposition}\cite{Brad} \label{divisors}
Let $M$ be a compact K\"ahler manifold and assume that condition \eqref{stability1} is satisfied. Pick any effective divisor $D$ on $M$ representing the homology class Poincar\'e-dual to $c_1(L)$. Then there is a solution $(A , \phi )$ of the vortex equations such that $D$ coincides with the divisor of the zero set of $\phi$ regarded as a holomorphic section of $L$. This solution is unique up to $U(1)$-gauge transformations. All vortex solutions  are obtained in this way. 
\end{proposition}
\noindent
Thus choosing a vortex solution is equivalent to picking an effective divisor of hypersurfaces on $M$. When $M$ is a Riemann surface, solutions are determined by a choice of $d$ points on $M$, where $d$ is the degree of $L$. In other words, they are determined by a point in the symmetric product $\text{Sym}^d M$. This description of the moduli space as a space of divisors is possible for vortices on line bundles, also called local vortices, but does not in general extend to other types of vortices.

\subsection{Case of simply connected manifolds}

\subsubsection*{Moduli space}

When the K\"ahler base $M$ is simply connected, besides describing the vortex moduli space as a space of divisors, it is easy to describe $\MM$ as a complex manifold. The arguments are all very standard.

\begin{proposition} \label{modulispace1}
Let $M$ be a $m$-dimensional, simply connected, compact K\"ahler manifold, and let $L \rightarrow M$ be a hermitian line bundle satisfying conditions \eqref{type_chern_class} and \eqref{stability1}. Denote by $H^0(M; L)$ the vector space of holomorphic sections of $L$ equipped with its unique holomorphic structure. Then the moduli space of vortex solutions is the projective space
\[
\MM \ \simeq \ \mathbb{P}\big( H^0(M; L) \big) \  .
\]
In particular, $\MM$ is empty if $L$ does not have any non-zero holomorphic sections. 
\end{proposition}
\begin{proof}
The usual short exact sequence of sheaves implies, after standard identifications (\cite{GH}),  the exact sequence in the cohomology
\[
H^{0, 1}(M, \CC) \xrightarrow{\ \ \ } \text{\rm Pic} (M) \xrightarrow{\ c_1\ } H^2(M, \mathbb{Z}) \xrightarrow{\text{\rm proj.}} H^{0,2}(M, \CC) \ .
\]
Since $M$ is simply connected, the space on the left is zero. Since $c_1 (L) \in H^{1,1} (M, \mathbb{C}) \cap H^{2} (M, \mathbb{Z})$, the exactness of the sequence implies that there is precisely one holomorphic structure on $L$. Scalar multiplication of a section $\phi$ by a constant $\lambda \in \CC^\ast$ corresponds to a (constant) complex gauge transformation, since these leave the holomorphic structure (or the Chern connection) invariant. So the moduli space of Bradlow pairs  ``holomorphic structure plus a non-zero holomorphic section" on $L$ is, up to isomorphism, $H^0 (M; L) \setminus \{ 0\}$ divided by the scalar action of $\CC^\ast$.
\end{proof}
\begin{remark}
The argument above does not really need the assumption that $M$ is simply connected, only that $b_1(M)=0$. Throughout the paper we could have used this weaker assumption on $M$.  
\end{remark}


\subsubsection*{Examples}

\noindent
{\bf (1)} When $M$ is a Riemann surface, vortex solutions on a degree $d$ line bundle $L \rightarrow M$ are determined by a choice of $d$ points on $M$, as described above. So the vortex moduli space $\MM$ is the symmetric product $\text{Sym}^d M$ --- a classic result.

\vspace{.4cm}

\noindent
{\bf (2)} Suppose that $M$ is a compact and simply connected K\"ahler manifold with $H^2(M; \mathbb{Z}) \simeq  \mathbb{Z}$. Then $\text{Pic} (M)$ is also isomorphic to $ \mathbb{Z}$. If the Picard group has a positive generator $E\rightarrow M$, the Chern class $c_1 (E)$ generates the additive group $H^2(M; \mathbb{Z})$ (this follows from the Lefschetz theorem on (1,1)-classes) 
and any holomorphic line bundle $L\rightarrow M$ is isomorphic to a tensor product $\otimes^d E$. The integer $d$ is called the degree of $L$ and is determined topologically by 
\begin{equation}\label{degree}
c_1 (L)\ =\ d\cdot c_1(E) \ .
\end{equation} 
Since $H^2 (M , \mathbb{R}) = \mathbb{R}$ is one dimensional, there is no need to take the parallel component in inequality \eqref{stability1}. That stability condition can be rewritten simply as
\[
c_1(L) \ < \ \frac{e^2 \tau}{2 \pi \, m} \, [\omega_M]  \ .
\]
According to Proposition \ref{modulispace1}, the moduli space of vortex solutions on $L = E^d$ is a projective space. Its complex dimension is determined by the dimension of $H^0 (M; E^d)$,  which depends only on the manifold $M$ and on the degree $d$. For instance, when $M$ is the standard projective space $M = \CC \mathbb{P}^{m}$, we have $E = \mathcal{O} (1)$ and
\[
\dim_\CC H^0 (\CC \mathbb{P}^{m} ; E^d) \ = \ \frac{(m+d)!}{m! \, d!}\ .
\]
When $M$ is the Grassmannian $\text{Gr} (n, k)$ of $k$-planes in $\CC^n$, a generator of $\text{Pic} (M)$ is the pullback by a Pl\"ucker embedding  of the hyperplane bundle $\mathcal{O}(1)$ over projective space. Equivalently, $E$ is dual to the $k$-th wedge product of the rank-$k$ tautological bundle over $\text{Gr} (n, k)$.  A classic result then says that the space of holomorphic sections of $E^d$ has complex dimension (see \cite[p.93]{Sakane} and the references therein):
\[
\dim_\CC H^0 \big(\text{Gr} (n, k), \, E^d \big) \ = \  \prod^{n-k}_{i=1} \prod^{n}_{j= n-k+1} \frac{d+ j-i}{j -i} \ .
\]

$\ $

\noindent
{\bf (3)} Let $M$ be the Hirzebruch surface $F_k = \mathbb{P}(\mathcal{O}(-k) \oplus \CC)$. This is a bundle over $\CC \mathbb{P}^1$ with fibre $\CC \mathbb{P}^1$. Both the Picard group and the cohomology $H^2 (M, \mathbb{Z})$ are isomorphic to $\mathbb{Z} \oplus \mathbb{Z}$. The standard generators $[F]$ and $[C]$ of the 2-cohomology are Poincar\'e-dual, respectively, to a fibre of the bundle and to the canonical section of the bundle. Let $L_{a,b} \rightarrow F_k$ be the line bundle with Chern class $a [C] + b [F]$, where $a$ and $b$ are positive integers. Then well known identifications of holomorphic sections (see ch.V.2 of \cite{Hartshorne}) say that 
\[
H^0(F_k , L_{a,b}) \ \simeq \ H^0 \Big(S^a \big(\mathcal{O}(-k) \oplus \CC \big) \otimes \mathcal{O}(b) \Big) \ \simeq \  \bigoplus_{l=0}^a H^0 \big( \mathcal{O}(-k l + b ) \big) \ ,
\]
and so 
\[
 \dim_\CC  H^0(F_k , L_{a,b}) \ = \ \sum_{l=0}^a \, \text{\rm max}\{0,\,  b - kl +1\} \  .
\]

\subsubsection*{Universal bundle}

\noindent
One of the main objects of study in this paper are the vortex universal bundles over the cartesian product $M \times \mathcal{M} $. In the case of the simplest abelian vortex equations, the universal bundle is just a  line bundle
\[
\mathcal{L} \ \longrightarrow \  M \times \mathcal{M} \ .
\]
It can be constructed in at least two different ways. Regarding $\mathcal{M}$ as the space of effective divisors on $M$ with class $c_1(L)$, the universal bundle corresponds to the divisor $\mathcal{D}$ on the cartesian product $M \times \mathcal{M}$ consisting of the points $(x, D)$ such that $x$ belongs to the support of $D$. It is clear that, for every $D \in \mathcal{M}$, the restriction $\mathcal{L} |_{M \times \{D\}}$ is isomorphic to $[D]$ as a line bundle over $M$. In the second construction of $\mathcal{L}$, described for example in \cite{Perutz, Bap2005}, we regard $\mathcal{M}$ as being the quotient $\mathcal{V} / \mathcal{G}$ of the space of vortex solutions  by the group of unitary gauge transformations on $L$. Then $\mathcal{L}$ is obtained as the quotient bundle $(L \times \mathcal{V}) / \mathcal{G}$, where the unitary gauge transformations act both on $L$ and on $\mathcal{V}$. If $[A, \phi]$ denotes the gauge equivalence class of a vortex solution,  there exists a global section $\Phi$ of $\mathcal{L}$ that takes a point  $(x, [A, \phi])$ in $M \times \mathcal{M}$ to the point $[ \phi (x),  A, \phi]$ in the quotient $(L \times \mathcal{V}) / \mathcal{G}$. Since the section $\Phi$ vanishes exactly along the tautological divisor $\mathcal{D}$, we recognize that $c_1 (\mathcal{L}) = c_1 ([\mathcal{D}])$ and that both constructions are equivalent up to smooth isomorphism. The first construction induces a natural holomorphic structure on $\mathcal{L}$; the second induces a hermitian metric, inherited from the one on $L$, and a unitary connection $\mathcal{A}$ (see \cite{Perutz, Bap2005}). 

When the manifold $M$ is simply connected, and so $\MM$ is a projective space,  the universal bundle $\mathcal{L}$ is particularly easy to describe.

\begin{lemma}\label{u_bundle1}
Denote by $p_1$ and $p_2$ the projection from the product $M \times \mathcal{M}$ onto the first and second factor, respectively. Then $\mathcal{L}$ is isomorphic to  $p_1^\ast L \otimes p_2^\ast \, H$ and
has Chern class
\[
c_1(\mathcal{L}) \ = \  p_1^\ast \, c_1 (L) \ + \ p_2^\ast\, c_1(H)  \ ,
\]
where $H$ denotes the hyperplane bundle over projective space.
\end{lemma}
\begin{proof}
The cartesian product $M \times \MM$ is simply connected,  so the holomorphic structure on $\mathcal{L}$ is determined by its first Chern class. Since each factor has trivial 1-cohomology, necessarily
\[
c_1(\mathcal{L}) \ = \  p_1^\ast \alpha \ + \ p_2^\ast\, \beta 
\]
with $ \alpha \in H^2 (M, \mathbb{Z})$ and $\beta \in H^2 (\MM, \mathbb{Z})$. The class $\alpha$ is the first Chern class of the restriction of $\mathcal{L}$ to $M \times \{\text{point}\}$. By construction of $\mathcal{L}$ this restriction coincides with $L$, so we get that $\alpha = c_1(L)$. The class $\beta$ is the first Chern class of the restriction of $\mathcal{L}$ to $\{\text{point} \} \times \MM$, and in \cite{Bap2010} it was shown that it coincides with $c_1 (H)$.
\end{proof}

\subsection{Case of abelian varieties}

 \subsubsection*{Moduli space}

 Let $L \rightarrow M$ be a complex line bundle over an abelian variety of complex dimension $m$. Since $M$ is  a complex torus, it can be represented as a quotient $V / \Lambda$ of a $m$-dimensional complex vector space by a discrete lattice of maximal rank $2m$. Any integral basis $\{ \lambda_1 ,  \ldots , \lambda_{2m}\}$ for this lattice also spans the full vector space over the reals, so it defines dual real coordinates $\{x^1, \ldots, x^{2m}\}$ on $V$. The differential 1-forms $\{\dd x^1, \ldots, \dd x^{2m}\}$ then descend to the quotient $M$ and generate the integral cohomology ring $H^\ast (M, \mathbb{Z})$. It is a standard result that the basis of $\Lambda$ can be chosen such that the first Chern class of $L$ is simply
\begin{equation} \label{c-class-L}
c_1(L) \ = \ \sum_{j=1}^m \, \delta_j \, \dd x^j \wedge \dd x^{m+j} \,
\end{equation}
for some integer coefficients $\delta_j \in \mathbb{Z}$ \cite{GH}. Now, we want to determine the moduli space of vortex solutions on $L$. The easiest way to describe this space is to use the Hitchin-Kobayashi correspondence of \cite{Brad} and think of vortices as pairs ``holomorphic structure plus a non-zero holomorphic section on $L$". On the one hand, since $M$ is an abelian variety, the set of holomorphic structures on $L$ is a complex torus of the same dimension as $M$, which will be denoted by $\text{Pic}^{c_1(L)} M$.  On the other hand, by assumption $L$ is positive, so we know \cite{GH} that for any holomorphic structure on $L$ the dimension of the space of sections is
\begin{equation} \label{dimension-r}
r(L) \ := \ \dim_\CC H^0(M; L) \ = \ \prod_{j=1}^m \, \delta_j \ .
\end{equation}
An argument similar to the one used in Section 2.2 then shows that the space of Bradlow pairs on $L$ should be a projective bundle over $\text{Pic}^{c_1(L)} M$. The fibre above each point $L'$ in this Picard variety will be the projectivization of $H^0(M; L')$. But which bundle is this precisely? The answer can be stated as follows.

$\ $

\noindent 
The torus  $\text{Pic}^{c_1(L)} M$ is (non-canonically) isomorphic to the Picard group $\text{Pic}^{0} M$ of holomorphic line bundles over $M$ with trivial Chern class, usually called the dual variety $\hat{M}$.
Consider the Poincar\'e line bundle $\mathcal{P} \longrightarrow M \times \text{Pic}^{0} M$ with the usual normalization.  Denoting by $p_1$ and $p_2$ the projections from this cartesian product onto the first and second factors, respectively, we can define the push-forward bundle
\[
\hat{L} \ := \ p_{2\ast} (\mathcal{P} \otimes p_1^\ast L) \ \longrightarrow \  \text{Pic}^{0} M \ .
\]
This is a complex vector bundle of rank $r(L)$ and is called the Fourier-Mukai transform of $L$ (see \cite{BBR, Huybrechts}, for instance, for a description of these transforms in a much broader setting). Then  the space of Bradlow pairs on $L$ --- or, equivalently, the space of effective divisors on $M$ with class $c_1 (L)$ --- coincides with the projectivization of $\hat{L}$.
\begin{proposition}
Let $L \rightarrow M$ be a positive line bundle over an abelian variety and suppose that condition \eqref{stability1} is satisfied. Then the moduli space of vortex solutions on $L$ is a smooth bundle over the torus $\text{Pic}^{\,0}M$ with typical fibre $\CC \mathbb{P}^{\, r(L) -1}$. In fact, 
the moduli space coincides with the projectivization 
\[
\MM \ \simeq \ \mathbb{P} (\hat{L}) \ 
\]
of the Fourier-Mukai transform of $L$.
\end{proposition}
\begin{proof}
It is a consequence of Proposition \ref{divisors} and the fact that $H^1(M;L)$ vanishes for an ample line bundle on an abelian variety. See, for instance, the description in \cite[Sect. 5]{Brion} of the space of effective divisors on $M$ that are Poincar\'e-dual to the class $c_1(L)$.
\end{proof}

\subsubsection*{Cohomology of the moduli space}

Since $\MM$ is a projective bundle, there exists a natural ``hyperplane" line bundle $H \longrightarrow \MM$ and a natural class
\begin{equation} \label{eta_abelian_variety}
\eta \ := \ c_1(H) \qquad \in \ \ H^2(\MM ; \mathbb{Z})
\end{equation}
that generates the cohomology of the projective fibres of $\mathcal{M}$. By the Leray-Hirsch theorem, any cohomology class in $H^\ast (\MM ; \mathbb{Z})$ can be written in the form 
\[
\sum_{k=0}^{r(L) -1} \,  \eta^k \; \text{\rm proj}^\ast \gamma_k \ 
\]
for some classes $\gamma_k \in H^\ast ( \text{Pic}^{\, 0} M ; \mathbb{Z})$, where $\text{\rm proj}\!\!: \mathcal{M} \longrightarrow  \text{Pic}^{\, 0} M$ denotes the natural projection.
Multiplicatively, these cohomology classes are subject to the standard relation \cite{Bott-Tu}:
\[
\eta^{r(L)} \ + \sum_{k=1}^{r(L)}\,  c_k(\hat{L}) \, \eta^{r(L) - k} \ = \ 0 \ .
\]
So to obtain the ring structure on $H^\ast(\MM ; \mathbb{Z})$ we need to know the Chern classes of the Fourier-Mukai transform $\hat{L}$. These are determined by the following result.
\begin{lemma} \label{chern_class_transform}
The Chern character of the Fourier-Mukai transform $\hat{L}$ is  
\begin{equation}\label{c-character}
{\text \rm ch}\, (\hat{L}) \ = \ \prod_{k=1}^m \, (\delta_k \, - \, \dd x^\ast_k \wedge \dd x^\ast_{m+k}) \ ,
\end{equation}
where $\{ \dd x^\ast_k\} $ is the basis of $H^1(\hat{M}, \mathbb{Z})$ dual to the basis $\{ \dd x^k\}$ of $H^1(M, \mathbb{Z})$. The total Chern class can then be written as 
\begin{equation} \label{c-class-transform}
c(\hat{L}) \ =\     \exp\Big[  \, \sum_{k=1}^{m}\,  (-1)^{k(k-1)/2} \, \frac{r}{k} \, \Big( \frac{c_1}{r}\Big)^k \, \Big]   \ = \ \Big(\,  \frac{\exp \big[\arctan \big( c_1 / r \big) \big]}{\sqrt{1+ (c_1 / r)^2}} \,    \Big)^r \ ,
\end{equation}
where the integer $r$ is given by \eqref{dimension-r} and $c_1 =  c_1(\hat{L})  = {\text \rm ch}_1 (\hat{L})$ is the first Chern class.
\end{lemma}
\begin{proof}
Using the definition of $\hat{L}$, the Chern character can be expressed as
\begin{align*}
\text{ch}(\hat{L}) \ &= \ \text{ch} [p_{2\ast} (\mathcal{P} \otimes p_1^\ast L)] \ = \   p_{2\ast}\,  \text{ch} ( \mathcal{P} \otimes p_1^\ast L) \\
&= \ p_{2\ast}\, \big[ \text{ch}( \mathcal{P})  \cdot  p_1^\ast  \text{ch}(L) \big]
\ = \ p_{2\ast}\, \Big (\exp \big[c_1 ( \mathcal{P}) \big] \cdot  p_1^\ast  \exp \big[c_1 ( L) \big]  \Big) \ .
\end{align*}
Notice that in the second equality above we have used the Grothendieck-Riemann-Roch formula applied to the projection morphism $p_2$. The Todd classes of $M$ and $\text{Pic}^{\, 0} M$ do not appear in this formula because both manifolds are Lie groups, and hence have trivial tangent bundle. On abelian varieties the first Chern class of the Poincar\'e line bundle is given by the standard formula \cite[p.198]{Huybrechts}
\[
c_1(\mathcal P) \ = \  \sum_{\alpha =1}^{2m} \, \dd x^\alpha \wedge \dd x^\ast_{\alpha} \ .
\]
The Chern class $c_1(L)$ is given by \eqref{c-class-L}. The push forward map $p_{2\ast}$, on its turn, is just the standard integration  over the fibre $M$ of the projection $M \times \text{Pic}^{\, 0} M \longrightarrow  \text{Pic}^{\, 0} M$. In particular
\[
p_{2\ast} \, \Big( \gamma \wedge \dd x^1 \wedge  \cdots \wedge \dd x^{2m} \Big) \  =\  \gamma \ 
\]
for any  class $\gamma$ in $H^\ast (\text{Pic}^{\, 0}M ; \mathbb{Q})$. Using all these elements, an exercise in bookkeeping yields
\begin{align*}
\text{ch}_j (\hat{L}) \ &= \  p_{2\ast}\, \Big[ \,  \frac{c_1({\mathcal P})^{2j} }{(2j)!} \cdot \frac{p_1^\ast \, c_1 (L)^{m-j}}{(m-j)!} \, \Big] \\
&= \ (-1)^j \ \sum_{\{ J \in \sigma | \, \# J = j \}} \Big(  \prod_{i \in \{1 , \ldots , m\} \setminus J}  \delta_i  \Big) \, \Big(  \prod_{l\in J}  \dd x^\ast_l \wedge \dd x ^\ast_{m+l}  \Big) \ ,
\end{align*}
where we have denoted by ${ \sigma}$ the set of subsets of $\{ 1, \ldots , m  \}$. It is then easy to recognize that the formula above is equivalent to \eqref{c-character}, as desired. 

To obtain the total Chern class of $\hat{L}$, recall that it is related to the Chern character by the standard formulae 
\[
c(\hat{L}) \ = \ \prod_{k=1}^m (1 + \alpha_k) \,  \qquad \qquad  \text{ch} (\hat{L}) \ = \ \sum_{k=1}^m  \, e^{\alpha_k} 
\]
for some $\alpha_k \in H^2 (M, \mathbb{Z})$. In particular we can write
\begin{align}
c(\hat{L}) \ &= \ \exp \Big[  \sum_{k=1}^m  \, \log (1 + \alpha_k) \Big] \ = \ \exp \Big[  \sum_{k=1}^m \, \sum_{j=1}^{+\infty} (-1)^{j-1}  \frac{(\alpha_k)^j}{j}  \Big]   \nonumber  \\
 &= \ \exp \Big[   \sum_{j=1}^{+\infty}  (-1)^{j-1}  (j-1)! \, \text{ch}_j (\hat{L}) \Big] \ , \label{relation1}
\end{align}
where of course all these sums have only a finite number of nonzero terms. Using \eqref{c-character} it is straightforward to check the recursion property
\begin{equation}\label{recursion}
\text{ch}_j(\hat{L}) \ = \ \frac{(-1)^j}{j \, \delta_1 \cdots \delta_m} \, \text{ch}_1(\hat{L}) \, \text{ch}_{j-1}(\hat{L}) \  = \ \frac{(-1)^{(j+2)(j-1)/2}}{ j! \, (\delta_1 \cdots \delta_m)^{j-1}} \ \text{ch}_1(\hat{L})^j \ .
\end{equation}
Then the first equality in formula \eqref{c-class-transform} of the lemma can be obtained by substituting \eqref{recursion} into the relation \eqref{relation1}. To deduce the second equality in formula \eqref{c-class-transform} just divide the series
\begin{equation} \label{series}
\sum_{k=1}^{+\infty}\,  (-1)^{k(k-1)/2} \, \frac{r}{k} \, \Big( \frac{c_1}{r}\Big)^k 
\end{equation}
into a sum of the even and odd sub-series and, using the standard Maclaurin expansion of the functions $\log(1+x)$ and $\arctan (x)$, check that \eqref{series} coincides with the expansion of 
\[
- \frac{r}{2} \, \log \big[ 1+ (c_1 / r)^2 \big] \ + \ r \, \arctan (c_1 /r) \ .
\]
This ends the proof of the lemma.
\end{proof}

\subsubsection*{Universal bundle}

\noindent
We consider now the universal line bundle 
\[
\mathcal{L} \ \longrightarrow \  M \times \mathcal{M} \ 
\]
defined in Section 2.2. When $M$ is an abelian variety we have the following result.
\begin{lemma}\label{u_bundle_av}
Denote by $p_1$ and $p_2$ the projections from the product $M \times \mathcal{M}$ onto the first and second factors, respectively.  Call $\text{\rm proj}$ the projection $ \mathcal{M} \longrightarrow  \text{Pic}^{c_1(L)} \simeq \text{Pic}^{0} M$. Then $\mathcal{L}$ is isomorphic as a complex line bundle to the tensor product $(\text{id} \times \text{\rm proj})^\ast \mathcal P\, \otimes \, p_1^\ast \, L \,\otimes\, p_2^\ast \, H$. In particular
\[
c_1(\mathcal{L}) \ = \  p_1^\ast \, c_1 (L) \ + \ p_2^\ast\, \eta \ + \   (\text{id} \times \text{\rm proj})^\ast \, c_1(\mathcal{P}) \ .
\]
\end{lemma}
\begin{proof}
Regard $\mathcal{M}$ as the space of effective divisors on $M$ with class $c_1 (L)$. The definition of $\mathcal{L}$ in Section 2.2 says that $\mathcal{L} |_{M \times \{D\}} \simeq [D]$ as a holomorphic bundle, hence 
\[
c_1 (\mathcal{L})|_{M \times \{D\}} \  = \ c_1 (\mathcal{L}|_{M \times \{D\}}) \ = \ c_1 ([D]) \ = \ c_1 (L) \ 
\]
for any divisor $D \in \mathcal{M}$. In particular the restrictions of the complex line bundle $\mathcal{L} \,\otimes\, p_1^\ast L^{-1}$ to submanifolds of the form  $M \times \{D\}$ always have trivial first Chern class. Then the universal property of the Poincar\'e line bundle (see \cite[p. 84]{BBR}) guarantees the existence of a unique holomorphic map $f: \mathcal{M} \longrightarrow \text{Pic}^{0} M$ such that $\mathcal{L} \,\otimes\, p_1^\ast L^{-1}$ is isomorphic to $(\text{id} \times f)^\ast \mathcal P\, \otimes\, p_2^\ast\, \mathcal{N}$, where $\mathcal{N}$ is a line bundle over $\mathcal{M}$. Moreover, the map $f$ is determined by the formula
\[
f(D) \ :=\  \mathcal{L}\,\otimes\, p_1^\ast L^{-1} \  |_{M \times \{D\} } \ =\  [D] \,\otimes\, L^{-1}
\]
for any $D \in \mathcal{M}$, so it coincides with the natural projection $\text{\rm proj}$. So we have that
\[
c_1(\mathcal{L}) \ - \ p_1^\ast \, c_1 (L) \ =  \   (\text{id} \times \text{\rm proj})^\ast \, c_1(\mathcal{P}) \ + \  p_2^\ast \, c_1(\mathcal{N})\ .
\]
The normalization of $\mathcal{P}$ is such that its restriction to the subset $\{0 \} \times \text{Pic}^{0} M$ is trivial, so the line bundle $\mathcal{N}$ coincides with the restriction of $\mathcal{L}$ to the submanifold $\{ 0\} \times \mathcal{M}$. As shown in \cite{Bap2010}, this last restriction is isomorphic to the hyperplane line bundle $H \longrightarrow \mathcal{M}$ associated with the projectivization, and this proves the formula for $c_1(\mathcal{L})$. The expression for $\mathcal{L}$ as a tensor product follows automatically because complex line bundles are determined up to smooth isomorphism by their first Chern class.
\end{proof}

\section{Vortices in abelian GLSM}

\subsection{General results}

Let $M$ be a $m$-dimensional compact K\"ahler manifold and let $\rho$ be a linear representation of the $k$-torus on $\CC^n$. We denote by $Q_j^a$ the integral weights of the representation, where the indices run as $1\leq j \leq n$ and $1\leq a \leq k$. Assuming that $\rho$ is effective, the span of the weights $\{ Q_j \in \mathbb{Z}^k: j= 1, \ldots, n  \}$ over the real numbers is the full space $\mathbb{R}^k$.
Now take a principal bundle $P \rightarrow M$ with group $T^k$ and define the hermitian line bundles
$ L_j =  P\times_{\rho_j} \CC$ over the manifold $M$ associated with the restriction of $\rho$ to the $j$-th copy of $\CC$ inside $\CC^n$. With these conventions the vortex equations are written for pairs $(A, \phi)$, where $A$ is a connection on $P$ and $\phi$ is a section of the direct sum bundle $\oplus^n_{j=1} L_j$. They read
\begin{align} \label{v-eq}
&\bar{\partial}_A  \phi \ = \ 0      \\
& i \Lambda F_A \,+ \,  e^2   \Big[  \big( \: \sum_j   Q_j \: |\phi^j|^2  \: \big)  - \tau \Big] \ = \ 0   \nonumber \\
& F_A^{0,2} \ = \ 0 \  \nonumber ,
\end{align}
where the curvature $F_A$ has values in the Lie algebra $\mathfrak{t}^k = i \mathbb{R}^k$ and $\tau$ is a fixed constant in $\mathbb{R}^k$. As usual, solutions of the vortex equations will be minima of the Yang-Mills-Higgs energy of the GLSM. The moduli space of solutions to these equations depends on the principal bundle $P$ and on the values of the constants $Q_j$ and $e^2 \tau$. For instance, decomposing the  torus bundle as a product of circle bundles 
\[
P = \times_{a=1}^k \, P^a \ \longrightarrow \ M \ ,
\]
it is clear that the third vortex equation has solutions only if 
\begin{equation}\label{integrability}
c_1 (P^a) \ \in \ H^{1,1} (M, \CC) \qquad {\rm for \ all} \  \ \   0 \leq a \leq k \, .
\end{equation}
Moreover, using  that the volume form on $M$ is just $\omega_M^m / m!$ and using the identity
\[
 \Lambda F_A  \: \ \omega_M^{m}\ = \ 
m \, F_A \wedge \omega_M^{m-1}\  ,
\]
an integration of the second vortex equation over $(M , \omega_M)$ shows that another necessary condition for the existence of vortex solutions is that the vector 
\begin{equation} \label{stability_vector1}
\sigma \ := \ (\text{Vol}\, M) \, \tau \; - \; \frac{i}{e^2}\,  \int_M F_A \, \wedge\, \frac{\omega_M^{m-1}}{(m-1)!}
\end{equation}
in $\mathbb{R}^k$ can be written as a linear combination of weights $Q_j$ with non-negative scalar coefficients. In other words, the vector $\sigma$ must be in the closed cone
\[
\Delta \ := \ \Big\{  v \in \mathbb{R}^k: \  v = \sum_{j=1}^n\, \lambda^j \, Q_j \ \ \rm{with}\ \ \lambda^j \geq 0  \Big\} \ \subseteq  \ \mathbb{R}^k
\]
generated by the weights of the torus action. To write $\sigma$ in a more ``cohomological" way, denote by $c_1 (P)$ the vector in the cohomology $\oplus^k H^2 (M, \mathbb{Z})$ with components $c_1(P^a)$ --- a slight abuse of notation.
Then, using the standard $L^2$ inner product of forms to decompose $H^2(M, \mathbb{R})$ into the real line generated by $[\omega_M]$ and its orthogonal complement, we have that
\[
c_1(P) \wedge [\omega_M]^{m-1}\  =\  c_1(P)^\parallel \wedge [\omega_M]^{m-1}  \ = \  \frac{c_1(P)^\parallel}{[\omega_M]} \ [\omega_M]^{m}\ ,
\]
where $\parallel$ denotes the component parallel to $[\omega_M]$ and the first equality is justified by the Lefschetz decomposition \cite[Ch 0.7]{GH}.
So the vector \eqref{stability_vector1}  can be written more simply as 
\begin{equation} \label{stability_vector2}
\sigma \ = \  \tau \, \text{Vol}\, M  \ - \ \frac{2 \pi \, m}{e^2} \, \, \frac{c_1 (P)^{\parallel}}{[\omega_M]} \, \text{Vol}\, M  \ .
\end{equation}
This is a vector in $\mathbb{R}^k$ and, for each component, we are taking the part of $c_1 (P^a)$ parallel to $[\omega_M]$ in the cohomology $H^2(M, \mathbb{R})$. So the first observation is the following.
\begin{lemma}\label{necessary_condition}
The GLSM has vortex solutions only if condition \eqref{integrability} is satisfied and the vector $\sigma$ lies in the closed cone $\Delta \subseteq \mathbb{R}^k$ generated by the weights of the torus action.
\end{lemma}
\noindent
Assuming that this necessary condition holds, we now turn to the question of existence and construction of vortex solutions. Our main result here follows from a type of Hitchin-Kobayashi correspondence established in \cite{Banfield, MiR}. It allows us to construct full vortex solutions from pairs $(A, \phi)$ that satisfy only the first and third vortex equations. Before being more precise, let us introduce some notation. Given a section $\phi$ of the direct sum bundle $\oplus^n_{j=1} L_j$, we define the index subset
\[
I_\phi \ := \ \{ \, j \in \{1, \ldots, n \}: \, \,  \phi^j \not\equiv 0 \, \} \ .
\]
Then we call $S_\phi \subseteq \mathbb{R}^k$ the subspace spanned by the weights $\{ Q_j \in \mathbb{Z}^k: j \in I_\phi \}$, and $S_\phi^\perp$ the orthogonal complement with respect to the standard euclidean metric on $\mathbb{R}^k$. Finally, we define the cone in $S_\phi$ generated by linear combinations with non-negative coefficients
\begin{equation}\label{cone_phi}
\Delta_\phi \ := \ \Big\{  v \in \mathbb{R}^k: \  v = \sum_{j\in I_ \phi}\, \lambda^j \, Q_j \ \ \rm{with}\ \ \lambda^j \geq 0  \Big\} \ \  \subseteq  \ S_\phi \ .
\end{equation}
We will often talk about the interior of this cone. By this expression, we mean the interior of $\Delta_\phi$ regarded as a subset of $S_\phi$. Equivalently, the interior of $\Delta_\phi$ coincides with the subset of vectors that can be written as a linear combination of the weights $\{ Q_j \in \mathbb{Z}^k : j\in I_\phi \} $ with strictly positive scalar coefficients (see Lemma \ref{aux_lemma2} in the Appendix).
With these conventions in mind we can now state the main result of Section 3.1.
\begin{theorem}\label{Hitchin_Kobayashi}
Let $(A, \phi)$ be a pair satisfying the first and third vortex equations. There exists a gauge transformation $g : M \rightarrow (\CC^\ast)^k$ such that $(g(A), g(\phi))$ solves the full vortex equations if and only if the vector $\sigma$ defined in \eqref{stability_vector2} lies in the interior of the cone $\Delta_\phi$.  When this happens, the transformation $g$ is unique up to multiplication by $T^k$-gauge transformations, and by constant gauge transformations with values on the subgroup $\exp (S_\phi^\perp) \subseteq \mathbb{R}_+^k$ of the group $(\CC^\ast)^k$.  All vortex solutions  can be obtained in this way.
\end{theorem}
\noindent
As mentioned above, the proof of this result relies heavily on a type of Hitchin-Kobayashi correspondence established in \cite{Banfield, MiR}. This correspondence equates the existence of the gauge transformation $g$ required in Theorem \ref{Hitchin_Kobayashi} to a stability condition on the initial pair $(A, \phi)$. This stability condition is defined in \cite{Banfield, MiR} in abstract terms that apply to vortices in very general settings. It is not an easy condition to evaluate in practice, though. In the case of the abelian GLSM, however, the stability condition simplifies dramatically and has a very concrete and natural appearance. This is the content of the next result, which underlies the main part of the proof of Theorem \ref{Hitchin_Kobayashi}.
\begin{theorem}\label{tau-stability}
Let $(A, \phi)$ be a pair that satisfies the first and third vortex equations. It is a simple pair in the sense of \cite{MiR} if and only if the weights $\{ Q_j: \phi^j \not\equiv 0 \}$ span the whole $\mathbb{R}^k$. Such a simple pair is $\tau$-stable in the sense of \cite{MiR} if and only if the vector $\sigma$ defined in \eqref{stability_vector2} is in the interior of the cone $\Delta_\phi$. 
\end{theorem}

\noindent
An important feature of the moduli space of vortices on line bundles is that it can be described as a space of effective divisors on $M$, as we saw in Section 2.1. Unfortunately, for the general GLSM no such identifications exist.  A nice exception is the case of GLSM's with the special equality $n=k$, where again choosing a vortex solution corresponds to picking a set of effective divisors on $M$ that satisfy a natural topological condition. This is the content of the next result. To introduce the statement, observe that for an effective $T^k$-action on $\mathbb{C}^k$ the integral weights  $Q_1, \ldots , Q_k$ form a basis of $\mathbb{R}^k$. In particular, the vector $\sigma$ of \eqref{stability_vector2} can be uniquely written as a linear combination $\sum_{j=1}^k \lambda^j\, Q_j$. From Lemma \ref{necessary_condition} we know that if any of the coefficients $\lambda^j$ is negative, i.e. if $\sigma$ lies outside the cone $\Delta$, there are no vortex solutions. If $\sigma$ lies in $\Delta$, we denote by $I_+$ (respectively, $I_0$) the subset of $\{1, \ldots, k\}$ defined by the $j$'s such that $\lambda^j$ is positive (respectively, is zero). Then we can decompose $\mathbb{R}^k = S_0 \oplus S_+$, where the subspace $S_+$  (respectively, $S_0$) is spanned by the weights $Q_j$ such that $j\in I_+$ (respectively, $j\in I_0$). The vortex solutions are then described by the following result.

\begin{proposition}\label{divisors}
Assume that $n=k$, that the weights  $Q_1, \ldots , Q_k$ span $\mathbb{R}^k$, and that the vector $\sigma$ of \eqref{stability_vector2} lies in  the cone $\Delta \subset \mathbb{R}^k$. For each index $j \in I_0$ pick a holomorphic structure $\mathcal{H}_j$ on the line bundle $L_j$. For each  index $j \in I_+$ pick an effective divisor of analytic hypersurfaces $D_j$ on the manifold $M$ such that the fundamental homology class $[D_j]$ is Poincar\'e-dual to the Chern class $c_1 (L_j) =  \sum_a Q_j^a \, \,  c_1(P^a)$ in the cohomology $H^2 (M; {\mathbb Z})$. Then there exists a solution $(A, \phi)$ of the vortex equations \eqref{v-eq} such that:
\begin{itemize}

\item[{\bf (i)}] For all $j \in I_0$,  we have $\phi^j \equiv 0$ and the holomorphic structure on $L_j$ induced by $A$ coincides with $\mathcal{H}_j$.  

\item[{\bf (ii)}] For all $j \in I_+$, the divisor $D_j$ coincides with the divisor of the zero set of $\phi^j$ regarded as a section of $L_j$. 
\end{itemize}
This solution is unique up to $T^k$-gauge transformations, and up to constant gauge transformations with values on the subgroup $\exp (S_+^\perp) \subseteq \mathbb{R}_+^k$ of the group $(\CC^\ast)^k$. Different choices of divisors or holomorphic structures provide gauge inequivalent solutions. All vortex solutions are obtained in this way.
\end{proposition}
\begin{remark}
The principal torus bundle $P \rightarrow M$ may have classes $c_1(P^a)$ such that it is impossible to choose a set of divisors $D_j$ satisfying the required condition. The Proposition then asserts that such a GLSM has no smooth vortex solutions. The same thing happens if, for any $j\in I_0$, the line bundle $L_j$ does not admit  a holomorphic structure, i.e. if $L_j$ does not admit a connection with curvature of type $(1,1)$.
\end{remark}

\subsubsection*{Proofs}

\begin{lemma}\label{stabilizers}
Let $(A, \phi)$ be a pair that satisfies the first and third vortex equations. The infinitesimal gauge transformations that preserve $(A, \phi)$ are represented by the constant maps $M \rightarrow S_\phi^\perp \, \subseteq \mathbb{R}^k$. In particular, the stabilizer of $(A, \phi)$ under $T^k$-gauge transformations is discrete iff the weights $\{ Q_j: \phi^j \not\equiv 0 \}$ span the whole $\mathbb{R}^k$. If those weights generate the lattice $\mathbb{Z}^k$ over the integers, the stabilizer of $(A, \phi)$ is trivial.
\end{lemma}
\begin{proof}
An infinitesimal gauge transformation $v: M \rightarrow i \mathbb{R}^k$ acts on sections of $\oplus_j L_j$ as 
\[
(\phi_1 ,  \ldots , \phi_n ) \longmapsto \ \dd \rho_{v} (\phi) \ = \  \Big( \sum_a v_a\, Q^a_1\, \phi^1 ,\,  \ldots \, ,  \sum_a v_a \, Q^a_n \, \phi^n \Big) \ .
\]
So $\dd \rho_{v} (\phi)$ vanishes if and only if the sum $\sum_a v_a Q^a_j$ is identically zero for every index $j$ such that  $\phi^j \not\equiv 0 $. This is equivalent to saying that $v$ has values on the orthogonal complement $ S_\phi^\perp$. Moreover, since infinitesimal gauge transformations act on connections by $A \mapsto A + \dd v$, a map $v$ preserves the connection if and only if it is constant.

Suppose now that the weights $\{ Q_j : \phi^j \not\equiv 0  \}$ generate the full lattice $\mathbb{Z}^k \subset \mathbb{R}^k$ over the integers. Then the vector $(1, 0, \ldots , 0)$, for instance, can be written as a linear combination $\sum_{j\in I_\phi} b^j \, Q_j$ for some integers $b^j$. If a gauge transformation $g: M \rightarrow (\CC^\ast)^k$ preserves the section $\phi$, then we must have
\[
\prod^k_{a=1} (g_a)^{Q_j^a} \ = \ 1
\]
for all $j \in I_\phi$. This implies that
\[
1 \ = \ \prod_{j \in I_\phi} \prod^k_{a=1} (g_a)^{ b^j \, Q_j^a} \ = \ g_1\ .
\]
Using a similar argument for all the other components of $g$, we see that $g$ is the constant map to the identity in $(\CC^\ast)^k$.
\end{proof}

\begin{proof}[{\bf Proof of Theorem \ref{tau-stability}}]
In the language of \cite{MiR}, the pair $(A, \phi)$ is simple precisely if it has discrete stabilizer under gauge transformations. So the first statement of the theorem follows from Lemma \ref{stabilizers}. 
Now identify the Lie algebra $\mathfrak{t}^k$ with $\mathbb{R}^k$ so that the usual exponential map $\mathfrak{t}^k \rightarrow T^k$ is given by $v \mapsto e^{iv}$ for any $v \in \mathbb{R}^k$. Then the moment map $\mu : \CC^n \rightarrow \mathfrak{t}^k$ associated to the representation $\rho $ is 
\[
\mu (z) \ = \ - \, \frac{1}{2} \,\big( \: \sum_j   Q_j \: |z^j|^2   \ - \  \tau   \: \big) \ .
\]
Since the representation is effective, the image $\mu (\CC^n)$ is a $k$-dimensional cone inside the Lie algebra. 
In our setting, the definition of stability in \cite{MiR} can be stated as follows:

\noindent
{\it  A pair $(A, \phi )$ is $\tau$-stable if and only if
\begin{gather}  \label{stability2}
\int_{x \in M}  \Big[ \langle v, -i \Lambda F_A \rangle \ +\ 2 \, e^2 \lambda (\phi (x) , v) \, \Big] \, \frac{\omega_M^m}{m!}\quad
> \ \ \ 0 \quad \ {for\ all\ } v\in \RR^k \setminus \{ 0 \} \ .
\end{gather}}
\noindent
In this expression $\langle  \cdot , \cdot \rangle$ denotes the canonical euclidean product  in $\mathbb{R}^k$ and, calling $\eta_t^v : \CC^n \rightarrow \CC^n $ the gradient flow of the function $\langle v, \mu\rangle : \CC^n  \rightarrow \RR$, we define for any $z \in \CC^n$:
\[
\lambda (z , v) \ := \ \lim_{t\rightarrow +\infty} \ \langle v ,  \mu \rangle
(\eta_t^v (z))\ .
\]
Since the moment map $\mu$ is invariant by the representation $\rho$, so is $\lambda$, and the composition $\lambda (\phi (x) , v)$ is well defined. The integral of the first term in \eqref{stability2} is
\[
\int_{M}  \langle v, -i \Lambda F_A  \rangle \, \frac{\omega_M^m}{m!}\ \ = \ 
-\, 2 \pi\, m\, (\text{Vol}\, M) \, \frac{\langle  c_1(P)^{\parallel} , v \rangle}{[\omega_M]} \  .
\]
As for the second term, start by noticing that for abelian hamiltonian actions the gradient flow of $\langle v, \mu\rangle$ is related to the exponential map of the complexified group action: 
\[
\eta_t^v (z) \ = \ \rho_{\exp (- t v)} \, (z) \ .
\]
Using the explicit expressions for $\rho$ and the moment map, it is easy to check that
\[
\lambda (z ,v) \ = \ \begin{cases}
+ \infty  &\mbox{if for some $j$ we have $\langle Q_j , v \rangle < 0$ and $z^j \neq 0$}  \ , \\
 \langle  \tau , v \rangle / 2 & \mbox{otherwise} \ .
\end{cases}
\]
Consider the index subset $I_\phi = \{ j\in \{1, \ldots, n\} : \, \phi_j \not\equiv 0 \}$. Since the section $\phi$ is holomorphic, for every $j \in I_\phi$ the components $\phi_j$ are nonzero almost everywhere on $M$. It follows that
\[
\int_{x \in M}  2 \, e^2 \, \lambda (\phi (x) , v) \, \, \frac{\omega_M^m}{m!} \ = \   
\begin{cases}
+ \infty  &\mbox{if $\, \langle  Q_j , v \rangle < 0 \, $ for some $j \in I_\phi$} \ , \\
 e^2 \, (\text{Vol}\, M) \, \langle  \tau , v \rangle & \mbox{otherwise} \ .
\end{cases}
\]
In the first case the inequality in condition \eqref{stability2} is clearly satisfied. So we only need to worry about the second case, i.e. when $v$ is such that the inner product $\langle  Q_j , v \rangle$ is non-negative for all $j \in I_\phi$. In this case the inequality in condition \eqref{stability2} reduces to
\begin{equation}\label{stability3}
\big\langle  e^2 \, \sigma ,\ v \big\rangle \ > \ 0 \ .
\end{equation}
Now suppose that $\sigma$ can be written as a linear combination $\sigma = \sum_{j \in I} \alpha^j \, Q_j$ with strictly positive coefficients for some non-empty subset $I \subseteq I_\phi$. Then the inner product $\langle \sigma , v \rangle$ is non-negative. It vanishes precisely when $v$ is orthogonal to the weights $\{Q_j :  j \in I\}$, or, equivalently, when $v$ is orthogonal to the real span $S_I \subseteq \mathbb{R}^k$ of those weights.  But by assumption $\langle  Q_j , v \rangle$ is non-negative for all $j\in I_\phi$, i.e. the cone $\Delta_{\phi}$ lies entirely in one of the closed half-spaces of $\mathbb{R}^k$ determined by the plane $v^\perp$. So if $\langle \sigma , v \rangle$ vanishes,  the vector $\sigma \in S_I  \subseteq v^\perp$ necessarily lies on the boundary of the cone $\Delta_{\phi}$. This shows that whenever $\sigma$ is in the interior of  $\Delta_{\phi}$, the inequality \eqref{stability3} is always satisfied, and hence the pair $(A, \phi)$ is $\tau$-stable.

To prove the converse, suppose that $\sigma$ is outside the $k$-dimensional closed cone $\Delta_{\phi}$. Since the cone is convex, it has a $k-1$-dimensional facet such that $\sigma$ and the interior of the cone lie in opposite sides of the facet, or, more properly, lie in opposite sides of the subspace of $\mathbb{R}^k$ spanned by the facet.  Call this facet $F_\sigma$. On the other hand, if $\sigma$ lies on the boundary of the cone $\Delta_{\phi}$, choose $F_\sigma$ to be any $k-1$-dimensional facet that contains $\sigma$. Now let $v \in \mathbb{R}^k$ be a non-zero vector orthogonal to $F_\sigma$ and pointing towards the half-space where the cone lies. Then clearly $\langle u , v \rangle \geq 0$ for all $u \in \Delta_{\phi}$, and in particular $\langle  Q_j , v \rangle$ is non-negative for all $j \in I_\phi$. But at the same time $\langle \sigma ,\ v \rangle$ is negative or zero, so condition \eqref{stability3} is not satisfied. This shows that if $\sigma$ is not in the interior of $\Delta_{\phi}$, the pair $(A, \phi)$ is not $\tau$-stable. 
\end{proof}

\begin{proof}[{\bf Proof of Theorem \ref{Hitchin_Kobayashi}}]
The necessity part of the Theorem is clear: if the pair $(g(A), g(\phi))$ solves the full vortex equations, integration over $(M, \omega_M)$ of the second vortex equation shows that the vector $\sigma$ can be written as a combination of the weights $\{ Q_j: \phi^j \not\equiv 0 \} $ with strictly positive coefficients.  

The sufficiency and uniqueness follow from Theorem \ref{tau-stability} and the Hitchin-Kobayashi correspondence of \cite{Banfield, MiR} (see in particular Sections 6.2, 6.4 and 6.5 of \cite{MiR}).  Suppose first that the pair $(A, \phi)$ is simple in the sense of Theorem \ref{tau-stability}. Then the assumption on $\sigma$ implies that $(A, \phi)$ is $\tau$-stable, and so the existence of $g$ and its uniqueness up to $T^k$-gauge transformations follows directly from Sections 6.2 and 6.4 of \cite{MiR}. Now suppose that the subspace $S_{\phi} \subset \mathbb{R}^k$ spanned by the weights $\{ Q_j \in  \mathbb{Z}^k : \phi^j \not\equiv 0  \}$ has dimension $l< k$. Then the orthogonal decomposition $\mathbb{R}^k = S_{\phi} \oplus S_{\phi}^\perp$ allows us to decompose the second vortex equation as a pair of equations
\begin{gather}
i\Lambda F_A^\parallel \,+ \,  e^2   \Big[  \big( \: \sum_j   Q_j \: |\phi^j|^2  \: \big)  - \tau^\parallel \Big] \ = \ 0    \label{parallel}\\
 i \Lambda F_A^\perp \,- \,  e^2  \, \tau^\perp  \ = \ 0  \ ,    \label{orthogonal}
\end{gather}
where the symbols $\parallel$ and $\perp$ denote the components parallel and orthogonal, respectively, to the subspace $S_{\phi} \subset \mathbb{R}^k$. The natural strategy here is to solve these two equations one at a time. By Lemma \ref{stabilizers}, there are no infinitesimal gauge transformations of the form $M \rightarrow S_{\phi}$ that preserve the pair $(A, \phi)$. Using this fact and the assumption on $\sigma \in S_{\phi}$, the proof in Section 6 of \cite{MiR} guarantees the existence of a complex gauge transformation 
\begin{equation} \label{parallel_transformation}
g_\parallel: \ M \ \longrightarrow \  \exp (S_{\phi} \oplus i S_{\phi}) \ \subset \ (\CC^\ast)^k \ ,
\end{equation}
 unique up to $T^k$-gauge transformations, such that the pair $(g_\parallel(A), \, g_\parallel(\phi))$ solves equation \eqref{parallel} (see the comments in Section 6.5 of \cite{MiR}). Moreover, observe that any complex gauge transformation of the form \eqref{parallel_transformation} preserves equation \eqref{orthogonal}, and that any gauge transformation of the form
 \[
g_\perp: \ M \ \longrightarrow \  \exp (S_{\phi}^\perp \oplus i S_{\phi}^\perp) \ \subset \ (\CC^\ast)^k \ 
 \]
preserves equation \eqref{parallel}. So to find a solution of the full vortex equations, we only need to find a map $s : M \rightarrow S_{\phi}^\perp$ such that the imaginary gauge transformation $g_\perp = \exp s$ with values on $(\mathbb{R}_+)^k$ satisfies
\[
 i \Lambda F_{(\exp{s}) (A)}^\perp \,- \,  e^2  \, \tau^\perp  \ = \   \Delta_M (s) \ + \  i \Lambda F_{A}^\perp\,- \,  e^2  \, \tau^\perp \ = \ 0 \ ,
\]
where $\Delta_M$ denotes the Laplacian on the K\"ahler manifold $(M, \omega)$. But the assumptions of the Theorem say that the orthogonal component $\sigma^\perp$ vanishes, and by formula \eqref{stability_vector1} this is the same as saying that the integral of $ i \Lambda F_{A}^\perp\,- \,  e^2  \, \tau^\perp$ over the manifold $(M, \omega)$ vanishes. Hodge theory then guarantees the existence of the desired $s$, which is unique up to an additive constant in $S_{\phi}^\perp$. This completes the proof of the sufficiency and uniqueness statements. 

Finally, and very obviously, given any vortex solution $(A, \phi)$, the integral of the second vortex equation over $(M, \omega)$ shows that $\sigma$  must lie in the interior of the cone  $\Delta_\phi$. So the solution $(A, \phi)$ can be obtained by the method prescribed in the Theorem if we choose $g$ to be the trivial gauge transformation. In particular all vortex solutions can be obtained by the method prescribed in the Theorem.
\end{proof}

\begin{proof}[{\bf Proof of Proposition \ref{divisors}}]
We start by showing that each valid choice of divisors $D_j$ and holomorphic structures $\mathcal{H}_j$ leads to a vortex solution.
Firstly, the topological condition on the divisors and standard results on line bundles \cite{GH} guarantee that, for each $j \in I_+$, the nonzero divisor  $D_j$ defines a holomorphic structure on $L_j \rightarrow M$ and a holomorphic section $\phi^j \in \Gamma (L_j)$ such that $D_j$ is the divisor of the zero set of $\phi^j$.  In particular we have holomorphic structures on all the hermitian bundles $L_j$, and so for each index $j \in \{1, \ldots, k\}$ we can consider the Chern connection $B_j$ on $L_j$. This is a hermitian connection with curvature of type $(1,1)$. Then the system of $k$ equations and $k$ variables 
\[
B_j \ = \ \sum_a Q^a_j \, A_a \qquad \ \text{for all}\ \ 1 \leq j \leq k 
\]
has a unique solution $A$, since by assumption the vectors $Q_j$ form a basis of $\mathbb{R}^k$. Therefore the local 1-forms $A_1 , \ldots, A_k$ define a connection on the principal bundle $P \rightarrow M$. By construction we have that
\[
\bar{\partial}_A \phi \ = \ (\ldots , \bar{\partial}_{B_j} \phi^j ,  \dots  )_{1 \leq j \leq k} \ = \ 0  \ ,
\]
 and that
\[
\sum_a\, Q^a_j \, F_{A_a}^{0,2} \ = \ F_{B_j}^{0,2} \ = \ 0 \ .
\]
Since the $Q_j$'s form a basis of $\mathbb{R}^k$, the last expression implies that $F_{A}^{0,2}  = 0$. Thus the pair $(A, \phi)$ satisfies the first and third vortex equations. By construction, the index set $I_\phi$ coincides with $I_+$ and the vector $\sigma$ lies in the interior of $\Delta_\phi$. So we can apply Theorem \ref{Hitchin_Kobayashi} to obtain a complex gauge transformation $g:M \rightarrow (\CC^\ast)^k$ such that $(g(A), g(\phi))$ is a solution of the full vortex equations. Since complex gauge transformation do not change (the equivalence class of) the holomorphic structure on the line bundle, nor the divisor of the zero set of sections, we have obtained a full vortex solution satisfying $(i)$ and $(ii)$. 

The uniqueness part of the Proposition follows from the uniqueness in Theorem \ref{Hitchin_Kobayashi}. Let $(A, \phi)$ and $(A', \phi')$ be two vortex solutions satisfying conditions $(i)$ and $(ii)$. Then for each $j \in I_+$ the quotient $f_j = (\phi') ^j / \phi^j$ is a nowhere vanishing section of the trivial bundle over $M$. Moreover, for each $j \in I_0$, if we denote by $B_j$ and $B_j'$ the connections on $L_j$ induced by $A$ and $A'$, respectively, then assumption $(i)$ implies that $B_j$ and $B_j'$ are related by a complex gauge transformation $f_j : M \rightarrow \mathbb{C}^\ast$ on $L_j$. So we can define another complex gauge transformation $g:M \rightarrow (\CC^\ast)^k$ by the system of $k$ equations
\[
f_j \ = \ \sum_a \,  Q^a_j \, \,  g_a \qquad \ \text{for all}\ \ 1 \leq j \leq k \ .
\]
By construction we then have that $(A', \phi') = g(A, \phi)$ --- a fact that can be checked directly using the assumption that the vectors $Q_j$ form a basis of $\mathbb{R}^k$. The uniqueness part in Theorem \ref{Hitchin_Kobayashi} then guarantees that $g$ actually has values on the real torus $T^k$, or has at most a constant factor with values on the subgroup $\exp{S_\phi^\perp} \subseteq \mathbb{R}_+^k$ of the group $(\CC^\ast)^k$.  The uniqueness statement follows from the observation that $S_\phi = S_+$.

To see that all vortex solutions are obtained in this way, just observe that if $(A, \phi)$ is a vortex solution, for $j \in I_+$ we can take the divisors $D_j$ to be those defined by the zero set of $\phi^j$. Then the homology class of $D_j$ is Poincar\'e-dual to $c_1(L_j)$. For $j \in I_0$, we take the holomorphic structure $\mathcal{H}_j$ on $L_j$ to be the one determined by the connection on this bundle induced by $A$ (this connection is integrable because $A$ satisfies the third vortex equation). By the uniqueness proved above, the solution associated to these $D_j$'s and $\mathcal{H}_j$'s by the existence part of the Proposition will be gauge equivalent to the original $(A, \phi)$. Finally, different choices of $D_j$'s and $\mathcal{H}_j$'s provide gauge inequivalent solutions because $(\mathbb{C}^\ast)^k$-gauge transformations do not change (the equivalence class of) holomorphic structures on line bundles, nor the divisor of the zero set of sections.
\end{proof}

\subsection{Case of simply connected manifolds}

Assume that the compact K\"ahler manifold $M$ is also simply connected. Then line bundles over $M$ have unique holomorphic structures, as in Section 2.1. The vector space $V$ of all holomorphic sections of $\oplus_j L_j$ splits as a direct sum
\begin{equation} \label{space_sections}
V \ = \ \bigoplus_{j=1}^{n} H^0 (M , L_j)  \ . 
\end{equation}
If we let the torus $T^k$ act on each subspace $H^0 (M , L_j)$ by scalar multiplication with weights $Q_j \in \mathbb{Z}^k$, we get a global torus action on $V$. If we choose and fix an euclidean metric on each subspace $H^0 (M , L_j)$, we get a global metric on $V$ that is preserved by the torus action. A moment map $\nu: V \rightarrow i\mathbb{R}^k$ for this action is given by 
\begin{equation}\label{m_map}
\nu (w) \ = \ - \frac{i}{2} \, \Big ( \sum_{j=1}^n \, Q_j \, |w^{j}|^2  \ - \ \sigma \Big) \ ,
\end{equation}
where each $w^j$ is a vector in $H^0 (M , L_j)$.
With these definitions we can state the main result of this section.
\begin{proposition} \label{toric_moduli_space}
Let $M$ be a compact and simply connected K\"ahler manifold. Assume that condition \eqref{integrability} is satisfied. Then the moduli space of vortex solutions of the GLSM is isomorphic to the symplectic quotient 
\begin{equation} \label{moduli_space_glsm1}
\MM \ \simeq \ \nu^{-1}(0)\, /\, T^k .
\end{equation}
In particular, the vector $\sigma$ defined in \eqref{stability_vector2} determines the level of the quotient.
\end{proposition}
\begin{proof}
Fix a holomorphic structure on $P$ and let $A$ be the respective Chern connection. Since $M$ is simply connected, this holomorphic structure is unique up to isomorphism, so every integrable connection $A_1$ on $P$ is $(\CC^\ast)^k$-gauge equivalent to $A$. The gauge transformation that takes $A_1$ to $A$ is unique up to multiplication by constant $(\CC^\ast)^k$-gauge transformations. It follows  that every solution $(A_1, \phi_1)$ of the first and third vortex equations is $(\CC^\ast)^k$-gauge equivalent to some pair $(A, \phi)$, where $\phi$ is a holomorphic section of $\oplus_j L_j$ (equipped with the holomorphic structure induced by $P$) that is uniquely determined up to the $(\CC^\ast)^k$-action on $V$. So we have a natural map from the moduli space $\mathcal{B}$ of solutions of the first and third vortex equations to the topological quotient
\[
\gamma: \ \mathcal{B} \ \longrightarrow  \ V / (\CC^\ast)^k \ .
\]
This map is injective because two solutions $(A_1, \phi_1)$ and  $(A_2, \phi_2)$ that are $(\CC^\ast)^k$-gauge equivalent to the same pair $(A, \phi)$, are necessarily $(\CC^\ast)^k$-gauge equivalent to each other. There is also an obvious map $\MM \rightarrow \mathcal{B}$. This map is injective because the uniqueness part of Theorem \ref{Hitchin_Kobayashi} guarantees that two vortex solutions that are $(\CC^\ast)^k$-gauge equivalent, are in fact related by a $T^k$-gauge transformation. (Notice that the constant gauge transformations with values on the subgroups $\exp (S_\phi^\perp) \subseteq \mathbb{R}_+^k$ mentioned in  Theorem \ref{Hitchin_Kobayashi} preserve the respective vortex solutions, by Lemma \ref{stabilizers}). Moreover, it also follows from Theorem \ref{Hitchin_Kobayashi} that the image of the injective map $\MM \rightarrow \mathcal{B}$ is the set of equivalence classes of pairs $(A_1, \phi_1)$ such that the vector $\sigma$ is in the interior of the cone $\Delta_{\phi_1}$. Using the injective correspondence $\gamma$, the moduli space $\MM$ can thus be identified with the quotient of the $(\CC^\ast)^k$-invariant subset of holomorphic sections
\begin{equation}\label{stable_subset}
V_\sigma \ := \  \{ \phi \in V: \ \sigma\ {\text{is in the interior of the cone }} \Delta_\phi  \} \  \ \subseteq \ V 
\end{equation}
by the $(\CC^\ast)^k$-action.

Now, given any point $w \in V$, define the index subset $I_w$ as the set $\{j \in \{ 1, \ldots, n \}:\, w^j \neq 0 \}$. The orbit of $w$ under the $(\CC^\ast)^k$-action intersects the inverse image $\nu^{-1} (0)$ if and only if the vector $\sigma \in \mathbb{R}^k$ can be written as a linear combination $\sum_{j \in I_w} \lambda^j Q_j$ with strictly positive coefficients, i.e. if and only if $\sigma$ lies in the interior of the cone 
\[
\Delta_w \ := \ \Big\{  v \in \mathbb{R}^k: \  v = \sum_{j \in I_w }\, \lambda^j \, Q_j \ \ {\rm{with}}\ \ \lambda^j \geq 0  \Big\}  \ .
\]  
When this happens, the intersection of the $(\CC^\ast)^k$-orbit of $w$ with the inverse image $\nu^{-1} (0)$ is a single $T^k$-orbit (see for instance Lemma 7.2 in \cite{Kirwan}). This implies that
\begin{equation}\label{sympl-git-quotients}
\nu^{-1}(0)\, /\, T^k  \ \simeq \  \Big( (\CC^\ast)^k \cdot \nu^{-1} (0) \Big) \, /\, (\CC^\ast)^k \ \simeq \ V_\sigma  \, /\, (\CC^\ast)^k \ ,
\end{equation}
where $V_\sigma$ is the set of vectors $w \in V$ such that $\sigma$ is in the interior of the cone $\Delta_w$. But in the discussion above we had concluded that the last quotient can be identified with the vortex moduli space $\MM$, so this finishes the proof.
\end{proof}

\subsubsection*{Dimension and smoothness of the moduli space}

The dimension of the vortex moduli space 
depends on the value of the parameter $\sigma \in \mathbb{R}^k$; it can be described as follows. Let $\{ I_1, \ldots , I_N \}$ be the collection of all non-empty index subsets $I \subseteq \{1, \ldots, n\}$ such that $\sigma$ is in the interior of the cone $\Delta_{I}$. For each subset $I_s$ in the collection, define
the non-negative integers \[
D_s \ := \  \sum_{j \in I_s} \, \dim_\CC \, H^0(M, L_j)  
\]
and 
\[
d_s \ := \ \dim_{\mathbb{R}} \: \text{span}\{Q_j \in \mathbb{Z}^k: \: j \in I_s \} \ . 
\]
Then the dimension of $\MM$ is given by
\begin{equation}\label{dimension_moduli_space}
\dim_\CC \MM \ = \   \max \: \{ D_s - d_s \,: \:   1 \leq s \leq N \}.
\end{equation}
Observe that a generic vector $\sigma$ inside the cone $\Delta$ does not lie in any of the proper subspaces of $\mathbb{R}^k$ spanned by proper subsets of weights $Q_j$. This implies that for generic $\sigma$ the integer $d_s$ is equal to $k$ for all $1 \leq s \leq N$. In particular, the maximum in \eqref{dimension_moduli_space} is attained for the full index set $I = \{1, \ldots, n\}$, so the complex dimension of the moduli space in the generic case is just $\dim_\CC V - k$. 

For such generic values of $\sigma$ all vortex solutions have a discrete stabilizer under gauge transformations, as follows from Lemma \ref{stabilizers}.  The points in the subset $V_\sigma \subseteq V$ also have discrete stabilizers under the $(\CC^\ast)^k$-action. This means that for generic $\sigma \in \Delta$ the moduli space $\MM$ is a toric orbifold. 
To guarantee smoothness of the moduli space we need a stronger assumption on the weights $Q_j$ of the torus action. Suppose that:

{\it
\noindent
{\bf (C1)} The vector $\sigma$ defined in \eqref{stability_vector2} lies in the cone $\Delta$, but does not lie in any of the proper subspaces of $\mathbb{R}^k$ spanned by proper subsets of weights $Q_j$.

\noindent
{\bf (C2)} 
Every subset of weights in $\{Q_1, \ldots, Q_n \}$ that spans $\mathbb{R}^k$ also generates the integer lattice ${\mathbb Z}^k$ in $\RR^k$.
}

\noindent
Then it follows from Lemma \ref{stabilizers} that all vortex solutions have trivial stabilizer under gauge transformations. The points in $V_\sigma \subseteq V$ also have trivial stabilizers under the $(\CC^\ast)^k$-action, and so with these assumptions $\MM$ is a smooth toric manifold.


In the smooth case, the representation of the vortex moduli space as a symplectic quotient allows us to define a natural set of cohomology classes in $H^2 (\MM ; {\mathbb Z})$. In fact, observing that $\nu^{-1}(0) \rightarrow \MM$ is a smooth principal $T^k$-bundle, we can consider the associated line bundle 
\begin{equation} \label{associated_line_bundles}
H_a\ =\  \nu^{-1}(0) \times_{U(1)^a} \CC\ \longrightarrow\ \MM \ ,
\end{equation}
where the notation means that $T^k$ acts on $\CC$ by simple multiplication of the $a$-th $U(1)$-factor inside $T^k$. This is the same as the line bundle $V_\sigma \times_{(\CC^\ast)^a} \CC \longrightarrow \MM$ associated with the principal $(\CC^\ast)^k$-bundle $V_\sigma \rightarrow \MM$.
 We then define
\begin{equation} \label{definition_eta1}
\eta_a \ := \ c_1 (H_a)  \quad \quad{\rm in} \quad H^2 (\MM ; {\mathbb Z}) \quad {\rm for} \ \ a= 1, \ldots , k .
\end{equation}
These cohomology classes are standard in toric geometry and are known to generate the cohomology ring of the toric quotient $\MM$.

\subsubsection*{Universal bundle}

Abelian GLSM  deal with connections on a principal bundle $P \longrightarrow M$ with gauge group $T^k$. If we regard the moduli space $\MM$ as the quotient $\mathcal{V} / \mathcal{G}$ of the space of vortex solutions by the group of gauge transformations $\mathcal{G} = \text{\rm Map}(M \rightarrow T^k)$, then the quotient 
\[
\mathcal{U} :=  (P \times \mathcal{V}) / \mathcal{G} \ \longrightarrow \ M \times \MM 
\]
is also a $T^k$-bundle. It is called the universal bundle. The decomposition 
$P = \times^k_{a=1} \, P^a$ as a product of circle bundles induces a similar decomposition $\mathcal{U} = \times^k_{a=1}\, \mathcal{U}^a$.

In this paragraph we will describe the topology of the universal bundle $\mathcal{U}$, which is determined by the Chern classes $c_1(\mathcal{U}^a)$. We assume that conditions (C1) and (C2) of the preceeding paragraph are true, so that the moduli space $\MM$ is a smooth toric manifold described by Proposition \ref{toric_moduli_space}. 

\begin{lemma}\label{u_bundle1}
Denote by $p_1$ and $p_2$ the projections from the product $M \times \mathcal{M}$ onto the first and second factor, respectively. Then the first Chern classes of the universal circle bundles are
\[
c_1(\mathcal{U}^a) \ = \  p_1^\ast \, c_1 (P^a) \ + \ p_2^\ast\, \eta_a  \ 
\]
in the cohomology $H^2 (M \times \MM , \mathbb{Z})$, where the classes $\eta_a$ were defined in \eqref{definition_eta1}.
\end{lemma}
\begin{proof}
Both $M$ and the moduli space $\MM$ are simply connected, so the first Chern class of $\mathcal{U}^a$ is necessarily of the form
\[
c_1(\mathcal{U}^a) \ = \  p_1^\ast \, \gamma_M \ + \ p_2^\ast\, \gamma_\MM  
\] 
for some classes $\gamma_M$ in $H^2 (M , \mathbb{Z})$ and $\gamma_\MM$ in $H^2 (\MM , \mathbb{Z})$. The class $\gamma_M$ coincides with the first Chern class of the restriction of the bundle $\mathcal{U}^a$ to $M \times \{\rm point\}$. Choosing a particular vortex solution $(A, \phi)$, we have that
\[
\mathcal{U}^a \ |_{M \times \{ [A, \phi] \}} \ = \ \big(\, P^a \times [ (A, \phi)\cdot \mathcal{G}] \, \big) / \mathcal{G}  \ \simeq \  P^a \ , 
\]
where the last isomorphism follows from the fact that the $\mathcal{G}$-action on the space of vortex solutions is free (cf. Lemma \ref{stabilizers}). Thus $\gamma_M = c_1 (P^a)$. Similarly, the class $\gamma_\MM$ coincides with the first Chern class of the restriction of the bundle $\mathcal{U}^a$ to $\{\rm p\} \times \MM$, where $p$ is any point in $M$. To prove that this is just $\eta_a$, we will use an argument similar to the one in \cite{Bap2010}.  Start by considering the subgroup of gauge transformations
\[
\mathcal{G}_p \ := \ \big\{ g:\Sigma \rightarrow T^k \ {\rm such \ that \ } g(p)=(1, \ldots , 1) \ \big\} \ .
\label{4.8}
\]
It is clear that $\mathcal{G}$ splits as $\mathcal{G}_p \times T^k$. Denoting by $P_p$ the fibre of the principal bundle above the point $p$, we then have that
\begin{align}
\mathcal{U}^a \: |_{\MM \times \{ p\}} \ &= \ (\mathcal{V} \times  P_p )\:  /\:  \mathcal{G} \ = \ [(\mathcal{V} \times P_p) / \mathcal{G}_p ] \: / \: T^k  
\label{4.9} \\
&= \ [(\mathcal{V} / \mathcal{G}_p ) \times P_p] \: / \: T^k \ \simeq \ \mathcal{V} \: / \: \mathcal{G}_p \ ,  \nonumber
\end{align}
where in the last step we have used that the $T^k$-action on the fibre $P_p$ is free and transitive. So we conclude that the restriction of $\mathcal{U}^a$ to $\MM \times \{ p\}$ is isomorphic to the $T^k$-bundle $\mathcal{V} / \mathcal{G}_p \longrightarrow \MM$.
Consider now the space of pairs that satisfy the first and third vortex equations: 
\[
\mathcal{B} \ = \ \{  (A, \phi ) : \   \bar{\partial}_A \phi =0  \ \ {\rm and}  \ \  F_A^{0,2} = 0 \}  \ .
\]
Define a gauge invariant subset $\mathcal{B}_\sigma^\ast \subset \mathcal{B}$ by only taking the pairs $(A , \phi)$ such that the vector $\sigma$ is in the positive cone generated by $\{Q_j : j \in I_\phi \}$, i.e. the pairs such that $\phi \in V_\sigma$, using the notation in \eqref{stable_subset}. 
Observe firstly that, by Lemma \ref{stabilizers} and assumptions (C1) and (C2), the group $\mathcal{G}^\CC$ of complex gauge transformations acts freely on $\mathcal{B}_\sigma^\ast$. Secondly, Lemma \ref{necessary_condition} says that all vortex solutions are in this subset, which implies that  $\mathcal{G}^\CC \cdot \mathcal{V} \subset \mathcal{B}_\sigma^\ast$. Finally, Theorem \ref{Hitchin_Kobayashi} guarantees that 
$ \mathcal{B}_\sigma^\ast \subset \mathcal{G}^\CC \cdot \mathcal{V}$, and hence that 
\begin{equation} \label{stable_set}
\mathcal{B}_\sigma^\ast \ = \  \mathcal{G}^\CC \cdot \mathcal{V} \ .
\end{equation}
It follows that  the vortex moduli space $\MM = \mathcal{V} /  \mathcal{G} =  \mathcal{B}_\sigma^\ast  /  \mathcal{G}^\CC$ can be identified with a complex quotient. Moreover  
\[
 \mathcal{B}_\sigma^\ast  /  \mathcal{G}_p^\CC = \   \big [ (\mathcal{G}^\CC / \mathcal{G}_p^\CC ) \cdot \mathcal{V}  \big ] \, /  \, \mathcal{G}_p  \ = \ [ (\CC^\ast)^k \cdot \mathcal{V}  ] \, / \, \mathcal{G}_p
\]
as a principal $(\CC^\ast)^k$-bundle over the moduli space. But the complex quotient $ \mathcal{B}  /  \mathcal{G}_p^\CC$ is just the space of holomorphic sections of $\oplus_j L_j$ up to isomorphism, or, in other words, it is the vector space $V$ of \eqref{space_sections}. Clearly, the subset $ \mathcal{B}_\sigma^\ast  /  \mathcal{G}_p^\CC$ is then the corresponding subset $V_\sigma \subset V$ defined in \eqref{stable_subset}. This means that the associated line bundles $H_a \rightarrow \MM$ defined in \eqref{associated_line_bundles} are just 
\[
H_a \ = \ V_\sigma \times_{(\CC^\ast)^a} \CC \  = \  [ (\CC^\ast)^k \cdot \mathcal{V}  ] \, / \, \mathcal{G}_p \ \times_{(\CC^\ast)^a}\: \CC \ \simeq \   \mathcal{V} / \mathcal{G}_p \ \times_{(U(1))^a}\: \CC \  \ .
\]
Finally, using the isomorphism established in the first part of the proof, we conclude that
\[
H_a \ \simeq \ \mathcal{U}\, |_{\{p\} \times \MM} \times_{U(1)^a} \CC \ , 
\]
and hence that $c_1(H_a) = c_1 ( \mathcal{U}^a\, |_{\{p\} \times \MM}) = \gamma_\MM$.
\end{proof}

\subsection{Case of abelian varieties}

As in the case of simply connected manifolds, the easiest way to describe the vortex moduli space when $M$ is an abelian variety is to make use of Theorem \ref{Hitchin_Kobayashi}. This result says that the moduli space $\MM$ is isomorphic to the complex moduli space of pairs $(A, \phi)$ satisfying the first and third vortex equations and the additional condition:

{\it
\noindent
{\bf (1)} The vector $\sigma$ of \eqref{stability_vector2} is in the interior of the positive cone $\Delta_\phi \subseteq \RR^k$ generated by the weights $Q_j \in \mathbb{Z}^k$ such that $\phi_j \not\equiv 0$ .
}

\noindent
Here $A$ is a connection on the $k$-torus bundle $P \rightarrow M$ and $\phi$ is a section of the associated bundle $\oplus^n_{j=1} L_j$ described in Section 3.1. Now, a pair $(A, \phi)$ satisfies the first and third vortex equations iff $A$ defines a holomorphic structure on $P$ and the section $\phi$ is holomorphic with respect to the induced holomorphic structure on $\oplus^n_{j=1} L_j$. The set of holomorphic structures on $P$ is (non-canonically) isomorphic to the $k$-fold cartesian product  $\times^k \,  \text{Pic}^0 M$ of the Picard group. Identifying the set of holomorphic structures on each associated line bundle  $L_j = P \times_{\rho_j} \CC$ with $\text{Pic}^0 M$, the group multiplication
\[
m_{\rho_j} : \times^k \,  \text{Pic}^0 M  \ \longrightarrow \  \text{Pic}^0 M  \qquad \qquad
(g_1 , \ldots , g_k)\  \longmapsto    \  \prod_{a=1}^k {(g_a)^{Q_j^a}}
\]
corresponds to the operation of inducing a holomorphic structure from $P$ to $L_j$. This is true because the tensor product of line bundles corresponds to multiplication on the Picard group. Thus for a given choice $g= (g_1 , \ldots , g_k)$ of holomorphic structure on $P$, we are looking for sections of $L_j$ with holomorphic structure $m_{\rho_j} (g)$, and from the discussion of Section 2.3 these define a vector space isomorphic to $(\hat{L_j})_{m_{\rho_j} (g)}$, the fibre of the Fourier-Mukai transform $\hat{L_j} \rightarrow \text{ Pic}^0 M$ above the point $m_{\rho_j} (g)$. Letting the point $g \in \times^k \,  \text{Pic}^0 M$ vary, we conclude that the space of pairs ``holomorphic structure on $P$ plus holomorphic section of $\oplus^n_{j=1} L_j$"  is isomorphic to the total space of the vector bundle
\begin{equation} \label{vector_bundle}
V \ := \ \bigoplus_{j=1}^{n} \ m_{\rho_j}^\ast  \, \hat{L_j} \ \longrightarrow \ \times^k \,  \text{Pic}^0 M \ .
\end{equation}
As in the case of simply connected manifolds, some of the points of $V$ still represent holomorphic sections related to each other by constant $(\CC^\ast)^k$-gauge transformations. In fact, the action $\rho_j$ of $(\CC^\ast)^k$ on the components $\phi_j$ of a section translates into scalar multiplication with weight $Q_j \in \mathbb{Z}^k$ on the fibres of each summand $\ m_{\rho_j}^\ast  \, \hat{L_j}$ of the bundle. This defines a global action of $(\CC^\ast)^k$ on the total space of the vector bundle $V$. The orbits of this action consist of pairs related by constant gauge transformations, and hence must be identified with a single point in the moduli space. So we must quotient $V$ by the $(\CC^\ast)^k$-action. Finally, Theorem \ref{Hitchin_Kobayashi} says that we should not quotient the whole $V$, but only the invariant subset $V_\sigma \subset V$ defined by the holomorphic  sections  $\phi$ that satisfy condition (1) above.
We are thus lead to the following result.
\begin{proposition} \label{toric_ms_abelian_varieties}
Let $M$ be an abelian variety and suppose that condition \eqref{integrability} is satisfied. Then the moduli space of vortex solutions of the GLSM is isomorphic to the complex quotient 
\begin{equation} \label{moduli_space_glsm2}
\MM \ \simeq \ V_\sigma  \; \big/  \; (\CC^\ast)^k  \ .
\end{equation}
In particular there is a natural bundle $\MM \longrightarrow \times^k \,  \text{Pic}^0 M$ whose typical fibre is a toric orbifold.
\end{proposition}
\begin{remark}
As mentioned above, the fibres of the the Fourier-Mukai transform $\hat{L_j}$ are isomorphic to a space of sections of $L_j = P \times_{\rho_j} \CC$. So a choice of hermitian metric on $L_j$ will induce a $L^2$ hermitian metric on $\hat{L_j}$. Doing this for every $j$ we get a hermitian metric on the vector bundle $V$. Using this metric we can consider the moment map $\nu: V \rightarrow \mathbb{R}^k$ defined by a formula  analogous to \eqref{m_map}. The moduli space $\MM$ can then be represented as a symplectic  quotient $\nu^{-1}(0) / T^k$, just as in the case of a simply connected $M$.
\end{remark}
\noindent
Note in passing that it is perfectly possible that the bundles $L_j$  do not have non-zero holomorphic sections. If this happens for all $j$, then the fibre of the vector bundle $V$ is zero-dimensional, i.e. the fibre is just a point. In this case the subset $V_\sigma$ will be empty whenever $\sigma \neq 0$, since condition (1) above cannot be satisfied, and hence also $\MM$ will be empty. If $\sigma =0$ the subset $V_\sigma$ will be equal to V, and hence the moduli space $\MM$ will be isomorphic to $\times^k \,  \text{Pic}^0 M$.

If the vector $\sigma$ has the generic values described in condition (C1) of Section 3.2, the $(\CC^\ast)^k$-quotient of Proposition \ref{toric_ms_abelian_varieties} will have complex dimension equal to $\dim_\CC (\text{fibre}\,V) + k (\dim_\CC M - 1)$. If condition (C2) is also satisfied, the moduli space $\MM$ will be a smooth manifold. In this case the representation of the vortex moduli space as a toric quotient allows us to define a natural set of cohomology classes in $H^2 (\MM ; {\mathbb Z})$. Just as in the case of simply connected $M$, described at the end of Section 3.2, we can consider the associated line bundles $H_a = V_\sigma \times_{(\CC^\ast)^a} \CC \longrightarrow \MM$  and define
\begin{equation} \label{definition_eta2}
\eta_a \ := \ c_1 (H_a)  \quad \quad{\rm in} \ \  H^2 (\MM ; {\mathbb Z}) \ \   {\rm for} \ \ a= 1, \ldots , k .
\end{equation}
These cohomology classes generate the cohomology ring of the fibres of $\MM \longrightarrow \times^k \,  \text{Pic}^0 M$. By the Leray-Hirsch theorem, the cohomology ring of $\MM$ is generated by the $\eta_a$'s and the pullbacks of classes from the base $ \times^k \,  \text{Pic}^0 M$ of the bundle.

\subsubsection*{Universal bundle}

In this paragraph we compute the Chern classes of the universal bundle $\mathcal{U} \rightarrow \ M \times \mathcal{M}$ in the case where $M$ is an abelian variety. The notation and assumptions are the same as in the case of simply connected $M$. In particular, we assume that conditions (C1) and (C2) of Section 3.2 are valid, so that the moduli space $\MM$ given by Proposition \ref{toric_ms_abelian_varieties} is a smooth manifold.

\begin{lemma}\label{u_bundle2}
Denote by $p_1$ and $p_2$ the projection from the product $M \times \mathcal{M}$ onto the first and second factor, respectively. Then the first Chern classes of the universal circle bundles are
\[
c_1(\mathcal{U}^a) \ = \  p_1^\ast \, c_1 (P^a) \ + \ p_2^\ast\, \eta_a  \ + \ (\text{\rm id} \times \text{\rm proj}_a)^\ast\, c_1(\mathcal{P})
\]
in the cohomology $H^2 (M \times \MM , \mathbb{Z})$, where the classes $\eta_a$ were defined in \eqref{definition_eta2}.
\end{lemma}
\begin{proof}
In this proof we will denote by the same symbol both the circle bundle $\mathcal{U}^a \rightarrow M \times \MM$ and the associated complex line bundle over the same base;  the same applies to the circle bundle $P^a \rightarrow M$. As was seen above, there are $k$  projections
\[
\text{\rm proj}_a \ : \ \MM \ \longrightarrow \ \times^k\, \text{\rm Pic}^0 M \ \longrightarrow \ \text{\rm Pic}^0 M \ ,
\] 
where the last arrow is the $a$-th projection from the cartesian product. In terms of vortex solutions, $\text{\rm proj}_a$ takes an equivalence class $[A, \phi]$ to the holomorphic structure on $P^a$ determined by the $a$-th component of the integrable $T^k$-connection $A$, which we will denote by $(P^a)_A$. Taking the tensor product $(P^a)_A \otimes (\tilde{P^a} )^{-1}$, where $\tilde{P^a}$ denotes a fixed holomorphic structure $P^a$, we get  and element of $\text{\rm Pic}^0 M$, as desired.

In terms of vortex solutions, the natural holomorphic structure on the line bundle $\mathcal{U}^a$ can be described in the following way. The proof of Lemma \ref{u_bundle1} says that $\mathcal{U}^a$ can be regarded as a complex quotient
\[
\mathcal{U}^a := (P^a \times \mathcal{V}) / \mathcal{G} \ = \ (P^a \times \mathcal{B}^\ast) / \mathcal{G_\CC} \ .
\]
The space $P^a \times \mathcal{B}^\ast$ has a natural complex structure determined by the one on $\mathcal{B}^\ast$ and by the rule that, on the tangent space $T_{p} P^a \oplus T_{[A, \phi]} \mathcal{B}^\ast$ at a given point, the complex structure on the first summand is the one of the holomorphic bundle $(P^a)_A$. This complex structure is preserved by complex gauge transformations, and hence induces the complex structure on the quotient $\mathcal{U}^a$.
Now, an argument just like in the proof of Lemma \ref{u_bundle1} says that the restriction of $\mathcal{U}^a$ to a submanifold of the form $M \times \{[ A, \phi]\}$ is 
\[
\mathcal{U}^a \ |_{M \times \{[ A, \phi]} \  = \ \big(\, (P^a)_A \, \times\, [(A, \phi) \cdot G_\CC] \,  \big) / \mathcal{G_\CC} \ \simeq \ (P^a)_A \ .
\]
In particular the restriction of $\mathcal{U}^a \otimes p_1^\ast (\tilde{P}^a)^{-1}$ to the same submanifold has trivial first Chern class. The universality property of the Poincar\'e bundle $\mathcal{P} \rightarrow M \times  \text{\rm Pic}^0 M$ then guarantees that the existence of a unique holomorphic map $f_a: \MM \rightarrow \text{\rm Pic}^0 M$ such that the bundle $\mathcal{U}^a \otimes p_1^\ast (\tilde{P}^a)^{-1}$ is isomorphic to $(\text{\rm id} \times f_a)^\ast \mathcal{P} \, \otimes \, p_2^\ast \, \mathcal{N}_a$, where $\mathcal{N}_a$ is a circle bundle over $\MM$. Moreover, the map $f_a$ is determined by 
\[
f_a \big( [A, \phi] \big) \ := \  \mathcal{U}^a \otimes p_1^\ast (\tilde{P}^a)^{-1} \ |_{M \times \{[ A, \phi]\}} \ = \ (P^a)_A \,\otimes\, (\tilde{P^a} )^{-1} \ = \  \text{\rm proj}_a \big( [A, \phi] \big)  \ .
\]
So we have that
\[
c_1 (\mathcal{U}^a) \ - \  p_1^\ast\, c_1(\tilde{P}^a) \ = \   (\text{\rm id} \times \text{\rm proj}_a)^\ast\, c_1(\mathcal{P}) \ + \   p_2^\ast \, c_1 (\mathcal{N}_a) \ .
\]
The normalization of $\mathcal{P}$ is such that its restriction to the subset $\{0 \} \times \text{Pic}^{0} M$ is trivial, so the line bundle $\mathcal{N}_a$ coincides with the restriction of $\mathcal{U}^a$ to the submanifold $\{ 0\} \times \mathcal{M}$. 
An argument similar to one in the proof of Lemma \ref{u_bundle1} then says that this last restriction is isomorphic to the complex line bundle $H_a \rightarrow \mathcal{M}$. This justifies the formula for $c_1(\mathcal{U}^a)$. 
\end{proof}

\section{K\"ahler class and volume of the $L^2$-metric}

\subsection{Line bundles over simply connected manifolds}

In Section 2 we have seen that the moduli space $\MM$ of abelian vortices on a positive line bundle $L \rightarrow M$ is a projective space whenever the base manifold is simply connected. Since the cohomology ring of projective spaces is very simple, in this instance it is easy to calculate the volume of $\MM$ once we know the K\"ahler class of the $L^2$-metric  \eqref{vortexmetric}. This class can be computed using a formula of Perutz \cite{Perutz}, as generalized in \cite{BS, Bap2010}, that expresses the K\"ahler form on $\MM$ in terms of the curvature $F_{\A}$ of the universal connection on the universal line bundle $\mathcal{L} \rightarrow \MM \times M$. The formula says that the K\"ahler form $\omega_\MM$ of the $L^2$-metric is given by the fibre integral
\begin{equation} \label{L2_form}
\omega_\MM \ = \ \int_M \frac{i \tau}{2 \, m!} \, F_{\A} \wedge \omega_M^m \ + \ \frac{1}{4 e^2 (m-1)!}\, F_{\A} \wedge F_{\A} \wedge \omega_M^{m-1} \ .
\end{equation}
(In passing, we note the remarkable similarity between the integral above and the integral in \eqref{energy_vortex}, a similarity for which the author cannot offer a good explanation.) 
Taking the cohomology class on both sides of the equation, we get that
\begin{equation}\label{L2_form_coh}
\big[ \omega_\MM \big] \ = \ \int_M \ \frac{\pi\, \tau}{m!} \, c_1(\LL) \wedge \big[ \omega_M \big]^m \ - \ \frac{\pi^2}{e^2 (m-1)!}\,   c_1(\LL) \wedge  c_1(\LL)  \wedge \big[ \omega_M \big]^{m-1} \ 
\end{equation}
in $H^2(\MM , \mathbb{Z})$. But we know what the Chern class $c_1 ({\mathcal L})$ is in the case of simply connected manifolds: it is given by Lemma \ref{u_bundle1}.
Using that formula, the fact that $\omega_M^m / m!$ is the volume form on $M$, and the fact that, by Lefschetz's decomposition,
\[
c_1(L) \wedge [\omega_M]^{m-1}\  =\  c_1(L)^\parallel \wedge [\omega_M]^{m-1} \ ,
\]
the fibre integration in \eqref{L2_form_coh} can be performed to yield 
\begin{equation} \label{k_class1}
\big[ \omega_\MM \big] \ = \  \pi \, \sigma \, \, c_1 (H) \ ,
\end{equation}
where $\sigma$ is the stability parameter \eqref{stability1} and $H \rightarrow \MM$ is the hyperplane bundle over projective space.  Observe how $\sigma$ effectively parametrizes the K\"ahler class of the $L^2$-metric on the moduli space.
So we get:
\begin{proposition}
Let $L\rightarrow M$ be a positive line bundle over a $m$-dimensional, simply connected, compact K\"ahler manifold. If condition \eqref{stability1} is satisfied, the volume of the $L^2$-metric on the vortex moduli space is 
\begin{equation*}
\text{\rm Vol}\, (\MM , \omega_\MM) \ = \ \frac{(\pi \, \sigma)^{\, r - 1}}{\big( r -1 \big)!} \ , 
\end{equation*}
where $\sigma$ is the positive number in \eqref{stability1} and $r$ is the complex dimension of $H^0(M, L)$. The total scalar curvature of the moduli space is 
\begin{equation} \label{scalar_curvature}
\int_\MM  s(\omega_\MM) \ {\rm vol}_\MM \ = \  \frac{2 \pi\, r\,  (r-1)}{\big[(r-1)! \big]^{1/(r-1)}} \ \big( \text{\rm Vol}\, \MM \big)^{(r-2)/(r-1)} \ .
\end{equation}
\end{proposition}
\begin{proof}
Since we already know the K\"ahler class on the moduli space $\MM \simeq \CC \mathbb{P}^{\, r(L) -1}$, the total volume of $(\MM , \omega_\MM)$ follows directly from the standard formula
\[
{\rm Vol}\,  \MM \ = \ \frac{1}{(\dim \MM)!} \int_\MM  [\omega_\MM]^{\dim \MM} \ .
\]
Similarly, the total scalar curvature follows from the equally standard formula
\[
\int_\MM  s(\omega_\MM) \ {\rm vol}_\MM \ = \ \frac{2\pi}{(\dim \MM -1)!} \int_\MM c_1(\MM) \wedge [\omega_\MM]^{\dim \MM -1} \  , 
\]
and the observation that for projective space $c_1(\MM) = (\dim \MM +1) \cdot c_1 (H)$.
\end{proof}

\subsubsection*{Observations}

\noindent
{\bf (1)} When $M$ is Riemann surface the volume and total scalar curvature of the vortex moduli space were computed in \cite{MN, Bap2010}.

$\ $

\noindent
{\bf (2)} When $M$ is a compact and simply connected K\"ahler manifold with $H^2(M; \mathbb{Z}) \simeq  \mathbb{Z}$, it is possible to express the parameter $\sigma$, and hence ${\rm Vol}\,  \MM$,  solely in terms of $\text{\rm Vol}\, M$, with no explicit reference to the K\"ahler class $[\omega_M]$. Using the notation of Example (2) in Section 2.2, we can write
\[
\frac{c_1(L)^{\parallel}}{[\omega_M]} \ =   \frac{c_1(L)}{[\omega_M]} \ = \ \frac{d}{\big( t_M^{-1}\, m! \, \text{\rm Vol}\, M \big)^{1/m}} \ ,
\]
where $d$ is the degree of $L$, $m$ is the complex dimension of $M$, and $t_M$ is the positive topological number defined by
\[
t_M \ := \ \int_M c_1(E)^m \ .
\]
Observe that when $M$ is a projective space, the number $t_M$ is equal to 1. When $M$ is the complex  Grassmannian, the generator $E$ of the Picard group is the pullback by a Pl\"ucker embedding $\text{Gr} (n, k) \rightarrow \CC \mathbb{P}^r$ of the hyperplane bundle over projective space. It is then easy to check that the topological number $t_M$ coincides the degree of the Grassmannian as a projective variety in $\CC \mathbb{P}^r$, which is known to be \cite[XIV.7]{Hodge-Pedoe}
\[
t_{\text{Gr} (n, k)} \ = \  \big(  k (n-k) \big)! \, \prod_{j =1}^k \frac{(j-1)!}{(n-k + j -1)!} \ .
\]
$\ $

\noindent
{\bf (3)}  Let $M$ be the Hirzebruch surface $F_k = \mathbb{P}(\mathcal{O}(-k) \oplus \CC)$. We will use the notation of Example (3) of Section 2.2. Consider the line bundle $L_{a,b} \longrightarrow F_k$ whose first Chern class is $a[C] + b [F]$ and write the K\"ahler class of the surface as $[\omega_M] = \lambda [F] + \delta [C]$. Then, using the standard intersection numbers of the divisors $[C]$ and $[F]$ (see ch.V.2 of \cite{Hartshorne}), we get that $\text{\rm Vol}\, M\, =\,  \delta \, (2 \lambda - k \delta ) /2$ and that
\begin{equation*}
 \frac{c_1(L_{a,b})^{\parallel}}{[\omega_M]}\, \,  \text{\rm Vol}\, M \ = \ \frac{1 }{2} \, \Big(a \lambda + b \delta (1 - k)  \Big) \ .
\end{equation*}

\subsection{Line bundles over abelian varieties}

When $M$ is an abelian variety, the moduli space $\MM$ is a projective bundle over the dual variety $\hat{M} = \text{\rm Pic}^0 M$, as was seen in Section 2.2. The K\"ahler class $[\omega_\MM]$ of the $L^2$-metric on the moduli space is given by formula \eqref{L2_form_coh}, just as in the case of simply connected manifolds. The only difference is that the first Chern class $c_1 (\mathcal{L})$ of the universal  bundle is now given by Lemma \ref{u_bundle_av}. Using this lemma, it is relatively straightforward to perform the fibre integral in \eqref{L2_form_coh} and hence obtain a formula for $[\omega_\MM]$.

\begin{proposition}
Let $L \rightarrow M$ be a positive line bundle over an $m$-dimensional abelian variety. If condition \eqref{stability1} is satisfied, the K\"ahler class of the $L^2$-metric on the vortex moduli space is
\begin{equation} \label{k_class2}
\big[ \omega_\MM \big] \ = \ \pi \, \sigma \: c_1 (H) \ - \ \frac{2\pi^2}{e^2} \, \text{\rm proj}^\ast \, \mathcal{F} \Big( \frac{[\omega_M^{m-1}]}{(m-1)!} \Big) \ ,
\end{equation} 
where $\sigma$ is the stability parameter defined in \eqref{stability1}, $H \rightarrow \MM$ is the hyperplane bundle in \eqref{eta_abelian_variety}, $\text{\rm proj}$ is the natural projection $\MM \rightarrow \hat{M}$, and $\mathcal{F}:  H^{2m-2}(M; \mathbb{R}) \rightarrow H^{2}(\hat{M}; \mathbb{R}) $ denotes the cohomological Fourier-Mukai transform.
\end{proposition}
\begin{proof}
We use expression \eqref{L2_form_coh} for the K\"ahler class $[\omega_\MM]$ and the formula 
\[
c_1 (\mathcal{L}) \ = \ p_1^\ast \, c_1 (L) \ + \ p_2^\ast\, \eta \ + \   (\text{id} \times \text{\rm proj})^\ast \, c_1(\mathcal{P})
\]
given in  Lemma \ref{u_bundle_av}. Start by observing that the first two terms in this formula coincide with the expression for $c_1 (\mathcal{L})$ in the case of simply connected manifolds. In particular when we substitute these two terms into \eqref{L2_form_coh} and perform the fibre integration we will obtain the result $\pi\, \sigma \, c_1 (H)$,  
just as in the case of simply connected manifolds (cf. formula \eqref{k_class1}). When we substitute the last term of $c_1 (\mathcal{L})$ into the right hand side of \eqref{L2_form_coh} we get the additional term
\[
- \: \frac{\pi^2}{e^2} \:  \int_M \: (\text{id} \times \text{\rm proj})^\ast \, \Big(c_1 (\mathcal{P}) \wedge c_1 (\mathcal{P}) \wedge \frac{[\omega_M^{m-1}]}{(m-1)!}\, \Big) \ = \ - \: \frac{2\pi^2}{e^2} \: \text{\rm proj}^\ast \int_M \: \text{\rm ch} (\mathcal{P}) \wedge \frac{[\omega_M^{m-1}]}{(m-1)!} \ .
\]
This term coincides with the last term in \eqref{k_class2}, by definition of $\mathcal{F}$.
\end{proof}

\begin{remark}
When $[\omega_M]$ is an integral class we can write $\omega_M = c_1(B)$ for some line bundle $B \longrightarrow M$. Then 
\[
\mathcal{F} \Big( \frac{[\omega_M^{m-1}]}{(m-1)!} \Big) \ = \ c_1 (\hat{B}), 
\]
where the vector bundle $\hat{B} \longrightarrow \hat{M}$ is the Fourier-Mukai transform of $B$.
\end{remark}
\begin{remark}
The dual variety $\hat{M} = \text{\rm Pic}^0 M$ coincides with the moduli space of  integrable $U(1)$-connections on $M$ with zero mean curvature, i.e. connections that satisfy $\Lambda F_A  = 0 =  F_A^{0,2}$. This moduli space has a natural $L^2$-metric defined by the first term in \eqref{vortexmetric}, i.e. by
\[
\Vert \dot{A} \Vert^2 \ = \ \int_M \ \frac{1}{4e^2} \: k_{ab} \: \dot{A}_1^a \wedge  \ast_{M} \, \dot{A}_2^b  
\]
Then the last term in \eqref{k_class2}, before pulling back by $\text{\rm proj}^\ast$, is exactly the K\"ahler class of this $L^2$-metric on the moduli space $\hat{M}$ of connections. 

\end{remark}

\noindent
Now that we know the K\"ahler class on the moduli space, the next task is to compute the total volume of $(\MM, \omega_\MM)$. Recall in the case of abelian varieties the moduli space is a projective bundle $\text{\rm proj}: \MM = \mathbb{P}(\hat{L}) \longrightarrow \hat{M}$. Therefore the integral of the volume form can be performed in two steps: first integrate over the projective fibres and then integrate over the base $\hat{M}$. That is
\[
\text{\rm Vol}\, \MM \ = \ \int_\MM \frac{[\omega_\MM]^{m+r-1}}{(m+r-1)!} \  = \   \int_{\hat{M}} \frac{ \text{\rm proj}_\ast \, [\omega_\MM]^{m+r-1}}{(m+r-1)!}  \ .
\]
Since $\MM$ is the projectivization $\mathbb{P}(\hat{L})$, 
the fibre-integration $\text{\rm proj}_\ast : H^\ast (\MM, \mathbb{R}) \rightarrow H^{\ast - r +1} (\hat{M}, \mathbb{R})$ can be conveniently expressed in terms of the Segre class of $\hat{L}$. This class is defined as the inverse of the total Chern class by the formal expansion of
\[
s(\hat{L}) \ = \ \frac{1}{c(\hat{L})} \ = \ \frac{1}{1+ c_1(\hat{L}) + c_2 (\hat{L}) +  \cdots} \ .
\]
As explained for instance in \cite{BDW}, it has the useful property that   
\[
\text{\rm proj}_\ast \, \big( c_1(H)^{l} \big) \ = \ s_{l-r+1} (\hat{L}) \ , 
\]
where $H \longrightarrow \MM$ is the hyperplane bundle and we are using the convention that $s_j (\hat{L})$ is zero for negative $j$. It is then straightforward to compute that
\begin{equation} \label{volume_moduli_space2}
\text{\rm Vol}\, (\MM , \, \omega_\MM) \ = \ \sum_{l=0}^{m}\: \frac{(\pi\, \sigma)^{\, l+r-1} \, (-2\pi^2 /e^2)^{m-l}}{(l+r-1)! \, (m-l)!} \  \int_{\hat{M}} s_l (\hat{L}) \: \wedge \:  \mathcal{F} \Big( \frac{[\omega_M^{m-1}]}{(m-1)!} \Big)^{m-l} \ ,
\end{equation}
To carry out the integration over the dual variety $\hat{M}$ one needs an explicit expression for the Segre class and for the Fourier-Mukai transform. The former can in principle be obtained by expanding the expression for $c(\hat{L})$ given in Lemma \ref{chern_class_transform}, although computationally this is non-trivial. The latter depends of course on the particular K\"ahler form $\omega_M$ that we take on $M$. Here we will not perform the integration over $\hat{M}$ in the general case. We will do it only for abelian varieties of complex dimension one and two.

\subsubsection*{Examples}

{\bf (1)} When $M$ has complex dimension one, i.e when it is an elliptic curve,  the natural identification $H^2(M; \mathbb{R}) = \mathbb{R}$ leads to the formulae $c_1 (L)^\parallel = c_1(L) = \deg L$ and $[\omega_M] = \text{\rm Vol}\, M$. Moreover, the Fourier-Mukai transform of 1 is just  $- \theta$, where $\theta$ stands for the positive generator of  $H^2(\hat{M}, \mathbb{Z})$. Hence our formula for the K\"ahler class coincides with the one given in \cite{MN, Perutz}.

$\ $

\noindent
{\bf (2)} Let $M$ have complex dimension $m=2$. Using the notation of Section 2.3, we can choose a set $\{ \dd x^k \}$ of generators of  $H^1(M, \mathbb{Z})$ such that
\[
c_1 (L) \ = \  \delta_1 \, \dd x^1 \wedge  \dd x^{m+1} \ + \  \delta_2 \, \dd x^2 \wedge  \dd x^{m+2} \ ,
\]
for some positive integers $\delta_k$. If we choose a K\"ahler form on $M$ of the form
\[
\omega_M \ = \  \lambda_1 \, \dd x^1 \wedge  \dd x^{m+1} \ + \  \lambda_2 \, \dd x^2 \wedge  \dd x^{m+2} \ ,
\]
where the $\lambda_k$'s are positive real numbers, then
\[
\text{\rm Vol} \ M \ = \ \lambda_1 \, \lambda_2 \ \ ,  \qquad \qquad \frac{c_1 (L)^\parallel}{[\omega_M]} \, \text{\rm Vol}\, M \ = \ \frac{1}{2}\, (\delta_1 \, \lambda_2 \ + \  \delta_2 \, \lambda_1) \ .
\]
The formulae of Lemma \ref{chern_class_transform} can be applied to obtain
\[
s(\hat{L}) \ = \ \frac{1}{c(\hat{L})} \ = \  1 \ - \ c_1 (\hat{L}) \ + \ \frac{r+ 1}{2r} \, \, c_1(\hat{L})^2 \ , 
\]
where the first Chern class of the Fourier-Mukai transform $\hat{L} \longrightarrow \hat{M}$ is 
\[
c_1 (\hat{L}) \ = \  -\, \delta_2 \, \dd x_1^\ast \wedge  \dd x^\ast_{m+1} \ - \  \delta_1 \, \dd x^\ast_2 \wedge  \dd x^\ast_{m+2} \ .
\]
In this expression the $ \dd x^\ast_{k}$'s are the generators of $H^1(\hat{M}, \mathbb{Z})$ dual to the $\dd x^{k}$'s. Similarly, the cohomological Fourier-Mukai transform of the K\"ahler class of $M$ is
\[
 \mathcal{F} ( [\omega_M]) \ = \ -\, \lambda_2 \, \dd x_1^\ast \wedge  \dd x^\ast_{m+1} \ - \  \lambda_1 \, \dd x^\ast_2 \wedge  \dd x^\ast_{m+2} \ .
\] 
The volume of the vortex moduli space can then be explicitly computed through \eqref{volume_moduli_space2}. The result is
\[
\text{\rm Vol} \, (\MM, \omega_\MM) \ = \  \pi \,  (\text{\rm Vol}\, M)\,  \Big[\: \tau \: + \:  \frac{4 \pi^2\,r}{e^4\, \sigma} \: \Big] \, \frac{(\pi\, \sigma)^{r}}{r!} \ ,
\]
where $r = r(L)= \delta_1 \delta_2$ and $\sigma$ is the parameter \eqref{stability1}.

\subsection{Abelian GLSM}

\subsubsection*{The K\"ahler class}

In this section we will compute the K\"ahler class of the metric on the moduli space for general abelian gauged linear sigma-models. This will be done both when $M$ is simply connected and when it is an abelian variety.  The essential ingredient is the abstract formula of Perutz for the  K\"ahler form $\omega_\MM$, or, more precisely, its generalization to gauged sigma-models and higher dimensional manifolds \cite{Bap2010}. This formula presents $\omega_\MM$ as the fibre integral
\begin{equation}
\omega_\MM  \ = \  \int_M \ \sum_{a=1}^k\ \frac{i\, \tau_a}{2\, m!}\, F_\mathcal{A}^a \wedge \omega_M^m \ + \ \frac{1}{4 e^2 (m-1)!}\,  F_\mathcal{A}^a \wedge F_\mathcal{A}^a \wedge \omega_M^{m-1} \ ,
\label{l2_metric}
\end{equation}
where $F_\mathcal{A}$ is the curvature on the universal bundle $\mathcal{U} \rightarrow M \times \mathcal{M}$. Since this is a principal $T^k$-bundle, the curvature is a 2-form on the base with values in $i \mathbb{R}^k$. Taking the cohomology class we then have
\begin{equation}
[\omega_\MM]  \ = \   \int_M \ \sum_{a=1}^k\  \frac{ \pi\, \tau_a}{m!}\, \, c_1(\mathcal{U}^a) \wedge \omega_M^m \ - \ \frac{\pi^2}{e^2 (m-1)!}\, \,  c_1(\mathcal{U}^a) \wedge c_1(\mathcal{U}^a) \wedge \omega_M^{m-1} \ . \label{l2_metric}
\end{equation}
When $M$ is simply connected the vortex moduli space is the toric orbifold given by Proposition \ref{toric_moduli_space}. We will assume that conditions (C1) and (C2) of Section 3.2 are satisfied, so that $\MM$ is smooth. Then the Chern classes $c_1(\mathcal{U}^a)$ are given explicitly by Lemma \ref{u_bundle1}. It is now straightforward to perform the fibre integration and obtain:
\begin{proposition} \label{kclass_glsm1}
Let $M$ a compact and simply connected K\"ahler manifold. The  K\"ahler class of the $L^2$-metric on the vortex moduli space \eqref{moduli_space_glsm1} is
\begin{equation} \label{kclass1}
\big[ \omega_\MM \big] \ = \  \sum_{a=1}^k  \, \pi\, \sigma^a \ \eta_a \ ,
\end{equation}
where $\sigma$ is the stability vector \eqref{stability_vector2} and  the classes $\eta_a \in H^{2}(\MM , \mathbb{Z})$ were defined in \eqref{definition_eta1}.
\end{proposition}
\noindent
When the manifold $M$ is an abelian variety, the moduli space $\MM$ is the toric fibration over $\text{\rm Pic}^0 (M)$ described in Proposition  \ref{toric_ms_abelian_varieties}. The Chern classes $c_1(\mathcal{U}^a)$ are in this case given by Lemma \ref{u_bundle2}. The fibre integration over $M$ can be performed in a way completely analogous to the case of vortices on line bundles (cf. proof of \eqref{k_class2}) to yield:
\begin{proposition} \label{kclass_glsm2}
Let $M$ be an abelian variety. The K\"ahler class of the $L^2$-metric on the vortex moduli space \eqref{moduli_space_glsm2} is
\[
\big[ \omega_\MM \big] \ = \  \sum_{a=1}^k \, \pi\, \sigma^a \ \eta_a  \ - \ \frac{2\pi^2}{e^2} \, \text{\rm proj}^\ast_a \, \, \mathcal{F} \Big( \frac{[\omega_M^{m-1}]}{(m-1)!} \Big) \ ,
\]
where $\sigma$ is the stability vector \eqref{stability_vector2}, the classes $\eta_a \in H^{2}(\MM , \mathbb{Z})$ were defined in \eqref{definition_eta2} and $\mathcal{F}$ is the cohomological Fourier-Mukai transform $\mathcal{F}:  H^{2m-2}(M; \mathbb{R}) \rightarrow H^{2}(\hat{M}; \mathbb{R})$.
\end{proposition}

\subsubsection*{Volumes of moduli spaces}

The volume of $(\MM, \omega_\MM)$ can in principle be computed using our knowledge of the K\"ahler class $[\omega_\MM]$ and of the cohomology ring $H^\ast (\MM , \mathbb{Z})$. For a general abelian GLSM, however, the intersection calculations in the cohomology ring are algebraically evolved, since the moduli space $\MM$ may be a complicated toric manifold or, in the case of abelian varieties,  may be an even more complicated toric fibration over the Picard group $\text{\rm Pic}^0 (M)$. So it is difficult to present a clean and general formula for the volume or total scalar curvature of the moduli space.
 In this paragraph we will carry out the computations only for the simplest examples. These will be GLSM with group $U(1)$ and $n$ sections $\phi_1 , \ldots , \phi_n$, where $U(1)$ acts with weight 1. As we will recall in Section 5, these models are closely related to holomorphic maps $M \rightarrow \CC \mathbb{P}^{n-1}$. We will perform the calculations in the case where $M$ is simply connected or when it is an abelian variety of complex dimension 2.

\noindent
{\bf (1)} Suppose that $P\rightarrow M$ is a $U(1)$-principal bundle over a simply connected manifold. Using the notation of Section 3, we will consider the associated GLSM with integral weights $Q_1 = \cdots = Q_n =1$. In this case the line bundles $L_1 , \ldots , L_n$ are all isomorphic, so the $\phi_j$'s can be regarded as sections the same line bundle, which we call simply $L$. Proposition \ref{toric_moduli_space} says that for this model the vortex moduli space is the projective space
\[
\MM \ \simeq \ \CC \mathbb{P}^{\, n\, r(L) \, -\, 1} \ ,
\]
where $r(L)$ denotes the complex dimension of the space of holomorphic sections $H^0 (M; L)$.
Proposition \ref{kclass_glsm1} says that the K\"ahler class is $[\omega_\MM] = \pi \, \sigma\, \eta$, where $\eta$ is the standard positive generator of the cohomology ring of projective space. Thus we have
\begin{equation}
\text{\rm Vol}\, (\MM , \omega_\MM) \ = \ \frac{(\pi \, \sigma)^{n\, r - 1}}{\big(n r -1 \big)!} \ .
\end{equation} 
The total scalar curvature of $(\MM, \omega_\MM)$ is still given by formula \eqref{scalar_curvature}. A few explict values of $r(L)$ for explict manifolds $M$ are given in the examples of Section 2.2.

\noindent
{\bf (2)} Consider the same GLSM as above, but now over an abelian variety $M$ of complex dimension 2, as is Example (2) of Section 4.2. The moduli space $\MM$ is in this case a projective bundle over the dual variety $\hat{M}$ with fibre $ \CC \mathbb{P}^{\, n\, r(L) \, -\, 1}$. Computations similar to those of Section 4.2 then lead to 
\[
\text{\rm Vol} \, (\MM, \omega_\MM) \ = \  \pi \,  (\text{\rm Vol}\, M)\,  \Big[\: \tau \: + \:  \frac{4 \pi^2\, n\, r}{e^4\, \sigma} \: \Big] \, \frac{(\pi\, \sigma)^{nr}}{(nr)!} \ .
\]

\subsection{Abelian GLSM with polynomial constraints}

In this paragraph we make a short digression to consider GLSM with polynomial constraints. These models are relevant to the study of holomorphic maps from $M$ to projective varieties \cite{Morrison-Plesser}.
Take the simple case of the gauged linear sigma-model with group $U(1)$ and weights $Q_1^1 = \cdots  = Q^1_n = 1$. Then all the complex line bundles $L_j$ are isomorphic to a given $L \rightarrow M$. The variables in the vortex equation are a connection $A$ and $n$ sections $\phi^1 , \ldots , \phi^n$ of this bundle. Now let $S(z^1, \ldots , z^n)$ be a complex homogeneous polynomial of degree $l$. It makes sense to write the equation
\begin{equation} \label{polynomial}
S(\phi^1, \ldots, \phi^n) \ = \ 0
\end{equation}
in the space of sections of the bundle $\otimes^{l} L$. We can thus look for vortex solutions of the model that satisfy the additional constraint \eqref{polynomial}. This will define a subspace $\MM^S$ of the moduli space $\MM$. 

When $L$ is positive and $\text{\rm Pic}(M) \simeq \mathbb{Z}$ and the stability condition \eqref{stability1} is satisfied, Proposition \ref{toric_moduli_space} states that the vortex moduli space for this GLSM is the projective space $\MM \simeq \CC\mathbb{P}^{\, n\, r(L) - 1}$. As usual, $r(L)$ stands for the complex dimension of the space of holomorphic sections of the bundle $L$ equipped with its unique holomorphic structure. In this case, it is clear that the subset $\MM^S$ cut out by condition \eqref{polynomial} will then be a projective variety inside $\MM$. In fact, choosing a basis of the $r(L)$-dimensional space $H^0(M, L)$ and a basis of the $r(\otimes^{l} L)$-dimensional space $H^0(M, \otimes^l L)$, the vector-valued polynomial equation \eqref{polynomial} can be rewritten as a set of $r(\otimes^{l} L)$ complex-valued homogeneous polynomial equations of degree $l$. The variables of these equations are the components of the sections $\phi^j$ with respect to the chosen basis of $H^0(M, L)$, which are precisely the homogeneous coordinates on the projective space $\MM$. In general the projective variety $\MM^S \subset \CC\mathbb{P}^{\, n\, r(L) - 1}$ may not be smooth. All depends on the polynomial $S$ and on the values $r(L)$. For a generic choice of $S$, however, the constrained moduli space $\MM^S \subset \MM$ is indeed a smooth projective variety of complex dimension $n\, r(L) - 1 - r(\otimes^{l} L)$. The Poincar\'e-dual of the homology class $[\MM^S]$ is 
\[
\text{\rm PD}([\MM^S]) \ = \ (l\, \eta)^{r(\otimes^{l} L)} \qquad \in \ \ H^{2\, r(\otimes^{l} L)} (\MM,  \mathbb{Z}) \ ,
\]
where $\eta$ is the standard positive generator of the cohomology of projective space. It follows that the  volume of the submanifold $\MM^S$ equipped with the metric induced by the $L^2$-metric on the moduli space $\MM$ is
\[
\text{\rm Vol}\, (\MM^S,\, \omega_\MM\,|_{\MM^S} ) \ = \ \frac{1}{(\dim \MM^S)!} \, \int_\MM \, (\omega_\MM)^{\dim \MM^S}\, \wedge \, \text{\rm PD}([\MM^S]) \  =\ \frac{(\pi \, \sigma)^{\dim \MM^S} \, l^{\, r(\otimes^{l} L)}}{(\dim \MM^S)!} \ .
\]

\section{Holomorphic maps to toric manifolds}

\subsection{Vortices and maps to projective space}

Let $M$ be a compact K\"ahler manifold and let $\mathcal{H}$ denote the (generally non-compact) space of holomorphic maps $f: M \rightarrow \CC \mathbb{P}^{n-1}$. Here we will describe a natural embedding of $\mathcal{H}$ into a compact vortex moduli space. The necessary ingredients are the Hitchin-Kobayashi correspondence of Section 3.1 and a standard construction in complex algebraic geometry. Everything here is well known, at least when $M$ is a Riemann surface \cite{BDW}. This exposition is a preamble to Section 5.2, where we study the space of maps $M \rightarrow X$ to more general toric targets. Nevertheless, the basic ideas already appear and are easier to grasp when the target is $\CC \mathbb{P}^{n-1}$. In the final Section 5.3 we make natural conjectures about the relation between the $L^2$-metrics on $\mathcal{H}$ and on $\MM$.

There is a well known construction \cite{GH} that identifies the space of maps $\mathcal{H}$ with a subset of the space $\MM$ of possible choices of ``a holomorphic line bundle $L \rightarrow M$ plus $n$ holomorphic sections, up to common rescaling of the sections". In fact, it is easy to see that such data $(L, \phi_1 , \ldots, \phi_n)$ determine a holomorphic map $f_{(L, \phi)}$ by the simple definition
\begin{align}\
f_{(L, \phi)} :\ M \ &\longrightarrow\ \CC \mathbb{P}^{n-1}  \label{induced_map}   \\
 x\  &\longmapsto \  \big[\phi_1 (x),  \ldots, \phi_n (x) \big] \ , \nonumber
\end{align}
where we are using homogeneous coordinates on the projective space. The map $f_{(L, \phi)}$ is globally well defined if and only if the sections $\phi_1 , \ldots, \phi_n$  do not vanish simultaneously at any point $x \in M$. In this case $f_{(L, \phi)}$ is clearly holomorphic, for the sections $\phi_j$ are all holomorphic. If we denote by 
$H \rightarrow \CC \mathbb{P}^{n-1}$ the hyperplane bundle, one can also  show that
\begin{equation} \label{pullback_coh}
f_{(L, \phi)}^\ast \, c_1 (H) \ = \ c_1 (L)
\end{equation}
in $H^2(M , \mathbb{Z})$. So the correspondence between holomorphic maps and systems of sections of line bundles preserves topological sectors. Now denote by $\mathcal{I} \subset \MM$ the (generally non-compact) subset defined be the condition that the sections $\phi_1 , \ldots, \phi_n$  do not vanish simultaneously anywhere on $M$. So far we have described a map $\mathcal{I} \rightarrow \mathcal{H}$. We will now describe the inverse $\mathcal{H} \rightarrow \mathcal{I}$. This implies that the correspondence $\mathcal{H} \hookrightarrow \MM$ is an injection with image $\mathcal{I}$. The inverse acts as follows: given a map $f \in \mathcal{H}$, just define $L := f^\ast H$ and take the $n$ holomorphic sections to be $\phi_j := f^\ast s_j$, where $s_1, \ldots , s_n$ are the elements of the standard basis of $H^0 (\CC \mathbb{P}^{n-1}; H)$ provided by the $n$ homogeneous coordinates on projective space. This correspondence has values in $\mathcal{I} \subset \MM$ and clearly is the inverse of the correspondence \eqref{induced_map}. 

Having described the map $\mathcal{H} \hookrightarrow \MM$, it is natural to ask whether $\mathcal{H}$ embeds as an open and dense subset of $\MM$. The answer is that this need not be the case. In fact, it is clear that $(L, \phi_1 , \ldots, \phi_n)$ belongs to the subset $\mathcal{I} \subset \MM$ precisely if the intersection $D_1 \cap \cdots \cap D_n$ of the associated zero-set divisors is empty. When $M$ is Riemann surface, this is generically true, since different divisors of points do not in general intersect. For higher-dimensional $M$, however, the divisors have support along hypersurfaces, and the generic intersection may not be empty. For instance, when $M$ is a projective space, we know that the codimension of the intersection of generic projective varieties is the sum of the codimensions of the varieties. This implies that for a generic choice of divisors the intersection $D_1 \cap \cdots \cap D_n$ is empty if and only if $n > \dim_\CC M$. So the space of holomorphic maps $\CC \mathbb{P}^m \rightarrow \CC \mathbb{P}^{n-1}$ of positive degree embeds as an open dense subset of $\MM$ iff $n > m$.

Note that so far in this Section we have defined $\MM$ to be the space  of possible choices of ``a holomorphic line bundle $L \rightarrow M$ plus $n$ holomorphic sections, up to common rescaling of the sections". However, as  was seen in Section 3.1, this space can be identified with a vortex moduli space. It coincides with the moduli space for the GLSM with goup $U(1)$ and integer weights  $Q_1^1 = \cdots = Q_n^1 = 1$, assuming that the stability parameter \eqref{stability_vector2} is positive.
Regarding $\MM$ as vortex moduli space has the advantage that it now comes equipped with a natural K\"ahler metric --- the metric defined in \eqref{vortexmetric} and studied in Section 4. Since the space $\mathcal{H}$ of maps also has a natural $L^2$-metric, one may wonder if there is any relation between these two metrics under the embedding $\mathcal{H} \hookrightarrow \MM $. This question was posed in \cite{Bap2010} in the case where $M$ is a Riemann surface. In Section 5.3 we extend the arguments and conjectures of that reference to higher-dimensional $M$.


\subsection{Vortices and maps to toric manifolds}

\subsubsection*{Toric manifolds}

Let $X$ be a smooth and not necessarily compact toric manifold of complex dimension $n-k$. We will think of it as a symplectic  quotient of $\CC^{n}$ by a linear, effective and hamiltonian $T^k$-action with integer weights $Q^a_j$, where the indices run as $a=1, \ldots , k$ and $j = 1 , \ldots , n$. This means that $X$ is the quotient space
\[
X \ :=\  \mu^{-1}(0) \, /\, T^k  \ , 
\]
with moment map $\mu: \CC^n \rightarrow  i \mathbb{R}^k$ given by
\begin{equation} \label{moment_map5}
\mu\, (z) \ = \ - \, \frac{i}{2} \Big( \sum_{j=1}^n \, Q_j\,  \vert z^j \vert^2 \ - \ \tau  \Big) \ .
\end{equation}
As in Section 3.1, denote by $\Delta$ the closed cone in $\RR^k$ defined by linear combinations of the weights $Q_j \in \mathbb{Z}^k$ with non-negative scalar coefficients. We will assume that the weights and the constant $\tau \in \mathbb{R}^k$ satisfy conditions similar to those of Section 3.2, namely:

{\it
\noindent
{\bf (H1)} The vector $\tau \in \RR^k$ lies in the cone $\Delta$, but does not lie in any of the proper subspaces of $\mathbb{R}^k$ spanned by proper subsets of weights $Q_j$.

\noindent
{\bf (H2)} 
Every subset of weights in $\{Q_1, \ldots, Q_n \}$ that spans $\mathbb{R}^k$ also generates the integer lattice ${\mathbb Z}^k$ in $\RR^k$.
}

\noindent
These two conditions guarantee that the quotient is indeed smooth of dimension $n-k$. One can also think of $X$ as a complex quotient of $\CC^{n}$. In this case one uses the complexified $(\CC^\ast)^k$-action on $\CC^{n}$ with the same integer weights, and considers the so-called stable subset $B_\tau :=  (\CC^\ast)^k \cdot \mu^{-1} (0)$ of the vector space $\CC^n$. This is clearly a $(\CC^\ast)^k$-invariant subset. It follows from condition (H1) and Lemma \ref{aux_lemma1} that $B_\tau$ is open and dense inside $\CC^{n}$. Moreover
\begin{equation}\label{complex_quotient}
X \ = \ B_\tau\, / (\CC^\ast)^k \ .
\end{equation}
All these observations are standard facts in toric geometry and can be recognized by arguments similar to those used in Lemma \ref{stabilizers} and Proposition \ref{toric_moduli_space}.

It is well known that the cohomology ring $H^\ast (X; \mathbb{Z})$ of such a toric manifold is generated by a set of natural $(1,1)$-classes. These are the first Chern classes of the  $k$ line bundles $H_a \rightarrow X$. Recall that $H_a$ is defined as  the quotient of the cartesian product $\mu^{-1}(0) \times \CC$ by the free $T^k$-action 
\begin{equation} \label{action5}
g \cdot (z_1,  \ldots, z_n; \: w) \ := \  \Big(\, z_1 \, \Big( \prod_{b=1}^k (g_b)^{Q_1^b} \Big) , \, \ldots ,\, z_n \, \Big( \prod_{b=1}^k (g_b)^{Q_n^b} \Big) \, ; \: g_a \,w \Big)  \ ,
\end{equation}
where $z_1, \ldots, z_n$ are complex coordinates on $\CC^n$, the ambient space of $\mu^{-1}(0)$.

\subsubsection*{Holomorphic maps as vortices}

Let $\mathcal{H}$ denote the space of holomorphic maps $f: M \rightarrow X$ with its usual topology. It can be broken into distinct topological components according  to the values of the pulback classes $f^\ast\, c_1(H_a)$ in the integer cohomology of $M$. For each vector 
\[
\xi \ \in \ \bigoplus^k \Big(\, H^2(M, \mathbb{Z}) \, \cap \, H^{(1,1)}(M, \CC) \, \Big) \ ,
\]
 we denote by  $\mathcal{H}_\xi$ the subspace of holomorphic maps $f$ such that $f^\ast\, c_1(H_a) = \xi_a$ for all indices $a$. 
Now let $P \rightarrow M$ be a principal $T^k$-bundle with $c_1(P) = \xi$. As in Section 3, we can consider the GLSM on the manifold $M$ associated to $P$ and the weights $Q^a_j$. The vortex equations for this theory are as in \eqref{v-eq}. It is clear from the definition \eqref{stability_vector2} of $\sigma$ that, for a sufficiently big value of the constant $e^2$, the assumptions (H1) and (H2) above imply the analogous conditions (C1) and (C2) of Section 3.2. So in this case the vortex equations will have the nice set of solutions described in that Section. Making explicit the Chern class of the principal bundle, we will denote by $\MM_\xi$ this vortex moduli space.

\begin{proposition}\label{embedding}
Assume that condition (H1) is satisfied. If the constant $e^2$ is sufficiently large there exists a natural holomorphic embedding $\mathcal{H}_\xi \hookrightarrow  \MM_\xi$. The image $\mathcal{I}_\xi \subset \MM_\xi$ of this map consists of the vortex solutions  $[A, \phi_1, \ldots, \phi_n]$ such that, for every point $x \in M$, the vector $\tau$ is in the interior of the cone $\Delta_{I_{\phi (x)}}$, where we have defined the index subset $I_{\phi (x)} = \{  j\in \{1, \ldots, n\} : \   \phi_j (x) \neq 0\}$.
\end{proposition}
\begin{remark}
It can happen that the subset $\mathcal{I}_\xi \subset \MM_\xi$ is empty, so in this case $\mathcal{H}_\xi $ will also be empty.
\end{remark}

\noindent
To describe the embedding $\mathcal{H}_\xi \hookrightarrow  \MM_\xi$, take any map $f \in \mathcal{H}_\xi$ and consider again the $k$ holomorphic line bundles $H_a \longrightarrow X$. (Each $H_a$ has a natural hermitian metric, so can be thought of as being a line bundle, a principal $\CC^\ast$-bundle, or a principal circle bundle; we will switch between these three viewpoints.) Defining $P_a := f^\ast H_a$, the cartesian product $P = \times^k P_a$ is a $T^k$-bundle over the manifold $M$ with first Chern class $c_1(P) = \xi$. Moreover, it is apparent from definition \eqref{action5} that the tensor product $\otimes_{a=1}^k \, (H_a)^{Q_j^a} \longrightarrow X$ is a line bundle with a natural section determined by the simple rule
\begin{equation} \label{natural_sections5}
s_j:\ [z_1,  \ldots, z_n] \ \longmapsto \ [z_1,  \ldots, z_n; \, z_j ] \ .
\end{equation}
This means that each associated bundle
\[
 L_j \ :=  \ \otimes_{a=1}^k \, (P_a)^{Q_j^a} \ = \ f^\ast \big( \otimes_{a=1}^k \, (H_a)^{Q_j^a}  \big) \ 
\]
has a natural holomorphic  section $\phi_j := f^\ast s_j$. So each map $f \in \mathcal{H}_\xi$ determines a holomorphic $(\CC^\ast)^k$-bundle $P \longrightarrow M$ and $n$ holomorphic sections $\phi_j$ of the associated bundles $L_j$. These correspond to a unique equivalence class of vortex solutions, i.e. to a point in $\MM_\xi$, provided that the Hitchin-Kobayashi correspondence of Theorem \ref{Hitchin_Kobayashi} is applicable, and this follows from the lemma below.

\begin{lemma}\label{index5}
Let $I_f$ be the index subset $\{ j\in \{1, \ldots, n\} : \   f^\ast s_j \not\equiv 0 \}$. Then for big enough $e^2$, the parameter \eqref{stability_vector2} can be written as a linear combination $\sum_{j \in I_f} \lambda^j \, Q_j$ with positive scalar coefficients, i.e. $\sigma$ is in the interior of the cone $\Delta_{I_f}$. 
\end{lemma}
\begin{proof}
For each index $j=1, \ldots, n$ denote by  $B_j$ the subspace of $\CC^n$ determined by the equation $z_j = 0$. Since our toric manifold is the quotient $X= \mu^{-1}(0) / T^k$, the intersection $\mu^{-1}(0) \cap B_j$ projects down to a (possibly empty) submanifold $X_j \subset X$. From definition \eqref{natural_sections5} it is  clear that the section $s_j$ vanishes precisely along this submanifold $X_j$. In particular the pulback section $f^\ast s_j$ is identically zero if and only if the image $f(M)$ is contained in $X_j$. This means that
\[
f(M)\ \subseteq \ \bigcap_{j \not\in I_f} X_j \ .
\]
A fairly uncontroversial consequence of this relation is that the set on the right-hand-side is not empty. Upstairs on $\CC^n$, this means that the set $\cap_{j \not\in I_f} B_j$ intersects $\mu^{-1}(0)$. Having a look at the definition \eqref{moment_map5} of the moment map, this clearly implies that there exists a  subset $I \subset I_f$  such that it is possible to write $\tau = \sum_{j \in I} \lambda'^j \, Q_j$ with positive scalar coefficients. The subset $I$ is not empty because of assumption (H1). By Lemma \ref{aux_lemma1} in the Appendix, it is also possible to write $\tau = \sum_{j \in I_f} \lambda''^j \, Q_j$ with positive scalar coefficients, so summing through the whole $I_f$. Finally, it is apparent from definition \eqref{stability_vector2} that for sufficiently large $e^2$ the stability parameter $\sigma$ is arbitrarily close to the value $\tau \, \text{\rm Vol}\,M$ in $\mathbb{R}^k$. Thus for big enough $e^2$ it is possible to write this parameter as a linear combination $\sigma = \sum_{j \in I_f} \lambda^j \, Q_j$ with positive scalar coefficients, as desired.
\end{proof}
\noindent
We have described above how to obtain a vortex solution in $\MM_\xi$ from a holomorphic map $f: M \longrightarrow X$. That solution is defined by putting $\phi_j = f^\ast s_j$. Now we have to show that it belongs to the image $\mathcal{I}_\xi$ as defined in the statement of Proposition \ref{embedding}.
\begin{lemma}
Assume condition (H1). Then the vector $\tau$ is in the interior of the cone $\Delta_{I_{f(x)}}$ for every point $x \in M$, where $I_{f(x)}$ is the index subset  $\{ j\in \{1, \ldots, n\} : \,   f^\ast s_j (x) \neq 0 \}$.
\end{lemma}
\begin{proof}
The proof is similar to the one of Lemma \ref{index5}. Using the notation and the arguments of that proof, we have that
\[
f(x)\ \in \ \bigcap_{j \not\in I_{f (x)}} X_j \ .
\]
Upstairs on $\CC^n$ this means that the set $\cap_{j \not\in I_{f (x)}} B_j$ intersects $\mu^{-1}(0)$. Condition (H1) and the definition \eqref{moment_map5} of the moment map then imply that there exists a non-empty subset $I \subseteq I_{f(x)}$ such that it is possible to write $\tau = \sum_{j \in I} \lambda'^j \, Q_j$ with positive scalar coefficients. By Lemma \ref{aux_lemma1} in the Appendix, the same is true for the whole $I_{f(x)}$.
\end{proof}

So far we have constructed a map $\mathcal{H}_\xi \rightarrow  \mathcal{I}_\xi$. Now we will describe its inverse, therefore showing that it is injective.

\begin{lemma}
 The map of Proposition \ref{embedding} has a holomorphic inverse  $\mathcal{I}_\xi \rightarrow  \mathcal{H}_\xi$ defined on its image.
\end{lemma}
\begin{proof}
As before, denote by $B_\tau \subset \CC^n$ the open, dense and $(\CC^\ast)^k$-invariant subset of points $(z_1, \ldots, z_n)$ such that the parameter $\tau$ is in the positive cone generated by $\{Q_j : j \in I_z  \}$, with $I_z := \{ j\in \{1, \ldots, n\} : \   z_j \neq 0 \}$. The representation \eqref{complex_quotient} of the toric manifold $X$ shows that there is a standard holomorphic projection $\pi : B_\tau \longrightarrow X$. A vortex solution $[A, \phi_1, \ldots, \phi_n]$ belongs to the subspace $\mathcal{I}_\xi$ precisely if the the images $(\phi_1 (x), \ldots, \phi_n (x))$ are in $B_\tau$ for all $x \in M$. So for any such vortex solution the composition $f_\phi := \pi \circ \phi$ is a map $M \longrightarrow X$. It is a holomorphic map because of the first vortex equation, i.e. because it is possible to find local trivializations of the line bundles such that the sections $\phi_j$ are all holomorphic. It is straightforward to check  that the line bundles $(f_\phi)^\ast H_a$ have Chern class $\xi_a$ and that, for all $x \in M$, we have 
\[
(f_\phi)^\ast s_j (x) \ =  \ \phi_j (x) \ . 
\]
Thus the image of $f_\phi$ by the map $\mathcal{H}_\xi \rightarrow  \mathcal{I}_\xi$ is gauge equivalent to the original vortex solution, and we have constructed the desired inverse holomorphic map $\mathcal{I}_\xi \rightarrow  \mathcal{H}_\xi$.
\end{proof}

\subsubsection*{Generic vortex solutions and holomorphic maps}

In this paragraph we ask whether the space $\mathcal{H}$ of holomorphic maps to toric targets embedds as an open and dense subset the vortex moduli space $\MM$. A general criterion for this to happen is presented. When the base $M$ is a projective space this criterion is much simpler. 

Consider the stable subset $B_\tau \subset \CC^n$. It can be described both as the set of orbits $(\CC^\ast)^k \cdot \mu^{-1} (0)$, or as the set of points $(z_1, \ldots, z_n)$ in $\CC^n$ such that the parameter $\tau$ is in the positive cone generated by $\{Q_j  \in \mathbb{R}^k :  z_j \neq 0\}$. If assumption (H1) is satisfied, it is a consequence of Lemma \ref{aux_lemma1} that $B_\tau$ is open and dense in $\CC^n$. The complement  $\CC^n \setminus B_\tau$ is a union of certain coordinate planes $E_\alpha$ in $\CC^n$ of different codimensions. Each of these planes is defined by a set of equations 
\begin{equation}\label{plane_equation}
z_{j_1 (\alpha)} \ =\  \cdots \ =\  z_{j_{n-l} (\alpha)} \ =\  0 \ ,
\end{equation}
 where $l = l (\alpha) \geq 0$ is the complex dimension of the plane.  Now,  a vortex solution $(A, \phi_1 ,  \ldots, \phi_n)$ is in the image $\mathcal{I}_\xi  \subset \MM_\xi$ of the embedding if and only if for every point $x\in M$ the value $\phi (x)$ is in the subset $B_\tau$; more precisely,  if and only if, for every plane $E_\alpha$, the section $\phi_{j_1 (\alpha)} \oplus \cdots \oplus \phi_{j_{n-l} (\alpha)}$ never vanishes. So we have:

\vspace{.2cm}

\noindent
{\it Represent the complement $\CC^n \setminus B_\tau$ as a union $\cup_\alpha E_\alpha$ of maximal planes defined by equations of the form \eqref{plane_equation}. Then the space of holomorphic maps $\mathcal{H}_\xi$ embeds as an open dense subset of the vortex moduli space $\MM_\xi $ if and only if, for all $\alpha$, the generic holomorphic section of $L_{j_1 (\alpha)} \oplus \cdots \oplus L_{j_{n-l} (\alpha)}$ never vanishes.}

\vspace{.2cm}

A more concrete result can be stated when the base $M$ is a projective space. To make this statement, consider again the maximal coordinate planes $E_\alpha$ in $\CC^n \setminus B_\tau$ and the corresponding sub-bundles $L_{j_1 (\alpha)} \oplus \cdots \oplus L_{j_{n-l} (\alpha)}$. To each such plane associate the integer $s_{\alpha}$ defined by the following two conditions:

\vspace{.2cm}

\noindent
{\bf (i)} $\ \ \  \; s_{\alpha} \ = \  - \infty  \qquad \quad   \text{if}\ L_{j_r (\alpha)}\  \text{is trivial for some}\ 1 \leq r  \leq n - \dim E_\alpha $ ;

\vspace{.2cm}

\noindent
{\bf (ii)} $\ \ \  s_{\alpha} \ = \ \dim E_\alpha \ + \ \# \big\{ 1 \leq r \leq  n- \dim E_\alpha : \ H^0(M; L_{j_r (\alpha)}) = 0 \big\} \qquad               \text{otherwise} \  $.

\vspace{.2cm}

\noindent
This integer is well defined, for when $M$ is simply connected the dimension of the space of holomorphic sections $H^0(M; L_{j_r})$ is an unambiguous characteristic of the only holomorphic structure on $L_{j_r}$. Now consider the maximum value
\begin{equation}
s \ := \ \max_\alpha  \big\{ \,  s_{\alpha} \,  \big\}  
\end{equation}
among all planes $E_\alpha$.  This maximal integer $s$ depends only on the constants $Q^a_j$ and $\tau$ that determine the quotient $X$, and on the cohomology class $\xi$ that determines the $L_j$'s. Then we have the following result.
\begin{proposition}\label{open_embedding_condition}
Assume that $M$ is a projective space and that (H1) is satisfied. Then the space of holomorphic maps $\mathcal{H}_\xi$ embeds as an open dense subset of the vortex moduli space $\MM_\xi $ if and only if  $n-s > \dim_\CC M$.
\end{proposition}
\begin{proof}
As stated before, a vortex solution $(A, \phi_1 ,  \ldots, \phi_n)$ is in the image $\mathcal{I}_\xi  \subset \MM_\xi$ of the embedding if and only if for every point $x\in M$ the value $\phi (x)$ is in the stable set $B_\tau$. Now consider a maximal coordinate plane $E_\alpha$ in the complement $\CC^n \setminus B_\tau$. Without loss of generality, we may assume that $E_\alpha$ is defined by the set of equations $z_{1} = \cdots = z_{{n-l(\alpha)}} = 0$, where $l(\alpha)$ denotes the dimension of $E_\alpha$. The solution $(A, \phi)$ belongs to $\mathcal{I}_\xi $ only if the  sections $\phi_{1}, \ldots, \phi_{{n-l(\alpha)}}$ do not have a common zero; otherwise, at that common zero, $\phi (x)$ would be in the plane $E_\alpha$, so outside $B_\tau$. Among these sections, a certain number will always be identically zero, because $H^0(M; L_{r(\alpha)})$ may be zero, i.e the bundle $L_{r(\alpha)}$ may have negative degree. This will happen for 
\[
m_{\alpha} \ := \ \# \big\{ 1 \leq r \leq  n- l(\alpha) : \ H^0(M; L_{r}) = 0 \big\} \ = \ s_{\alpha} \, - \, l(\alpha)
\]
sections. After relabeling if necessary, we can assume that the sections that are always zero are $\phi_{{n- s_\alpha +1}}, \ldots, \phi_{{n-l(\alpha)}}$. The remaining sections $\phi_{1}, \ldots, \phi_{{n-s_\alpha}}$
will generically be non-zero, and will vanish along a divisor in $M$ or will not vanish at all. So the condition that $\phi (x)$ lies outside the plane $E_\alpha$ for all $x \in M$ is reduced to the condition that the intersection of the corresponding zero-sets $Z (\phi_1) \cap \ldots \cap Z (\phi_{n-s_{\alpha}})$ is empty in $M$. If any of the bundles $L_1,  \ldots, L_{n-s_{\alpha}}$ is trivial, then its generic section will not vanish at all in $M$, and hence the intersection of the zero-sets will always be empty. If all the bundles $L_1,  \ldots, L_{n-s_{\alpha}}$ have positive degree -- the remaining case -- then the intersection of zero-sets will be an intersection of $n-s_{\alpha}$ divisors in $M$. When $M$ is a projective space, this last intersection is generically empty iff $n-s_{\alpha} > \dim_\CC M$. Thus a generic $\phi$ has values outside the plane $E_\alpha$ iff $n-s_{\alpha} > \dim_\CC M$. Doing the same for all other planes $E_\alpha$, we conclude that a generic $\phi$ has values outside the union $\cup_\alpha E_\alpha$ if and only if $n- s > \dim_\CC M$, where $s$ is the maximum of the $s_\alpha$'s.
\end{proof}
\begin{example}
When the target $X$ is also a projective space $\CC\mathbb{P}^{n-1} = (\CC^n - \{0\}) /  \CC^\ast$, there is only one plane $E$: the origin $\{0\}$. Thus $l_E = \dim E = 0$. As described in Section 5.1, the space $\mathcal{H}$ of holomorphic maps then embeds in the vortex moduli space of the GLSM with group $U(1)$ and integer weights $Q_1^1 = \cdots = Q_n^1 = 1$. In this case the line bundles $L_1, \ldots, L_n$ are all isomorphic, so if they are positive we have that $m_E = 0$ and $s = 0$. In this case $\mathcal{H}$ is open and dense in $\MM$ if and only if $n > \dim_\CC M$.
\end{example}

\subsection{Conjectures about  the $L^2$-metrics}

Consider again the spaces $\mathcal{H}_\xi$ of holomorphic maps $M \rightarrow X$ to toric targets. These spaces have a natural complex structure induced by the complex structures of the domain $M$ and of the target $X$. They also have a compatible K\"ahler metric $\omega_{\mathcal{H}}$ determined by the norm
\begin{equation}\label{l2_metric}
\| \dd f \|_{\omega_{\mathcal H}}^2 \ := \ \int_{(M, g_M)} |\dd f|^2 \ = \  \int_{(M, g_M)}  (g_M)^{\alpha \beta}\, (\partial_{\alpha} f^{j})\, (\partial_{\beta} f^{l})\, (g_{X})_{lj} \ .
\end{equation}
This is usually called the $L^2$-metric on $\mathcal{H}$. 
It is defined also for more general spaces of maps between Riemannian manifolds, not necessarily holomorphic maps. Since one can regard $\mathcal{H}$ as embedded in the vortex moduli space $\MM$, and the latter space has the natural $L^2$-metric $\omega_\MM$, it is natural to ask whether this is an isometric embedding or not. In \cite{Bap2010} it was argued that this should be the case if: {\bf (1)} one takes the metric $g_X$ on the toric target to be the symplectic quotient of the standard euclidean metric on $\CC^n$ determined by the moment map \eqref{moment_map5}, and,  {\bf (2)} one takes the strong coupling limit $e^2 \rightarrow \infty$. (Recall that both the vortex equations and the metric $\omega_\MM$ on the moduli space are $e^2$-dependent. The moduli space itself, i.e. the complex manifold $\MM$, is not, at least for $e^2$ large enough.) The heuristic arguments of that reference should hold also in the case of higher-dimension $M$, so we have:

\begin{conjecture}
Take the metric on $\mathcal{H}_\xi$ induced by the metric \eqref{vortexmetric} on the vortex moduli space through the embedding  $\mathcal{H}_\xi \hookrightarrow  \MM_\xi$ of Proposition \ref{embedding}. Then in the limit $e^2 \rightarrow \infty$ this metric converges to the natural $L^2$-metric \eqref{l2_metric} on $\mathcal{H}_\xi$.\footnote{This conjecture has now been proved in the case of maps from Riemann surfaces to projective space. The proof is presented in \cite{Liu}, and appeared on the arXiv after the first version of the present article.}
\end{conjecture}
\noindent
When $\mathcal{H}_\xi$ embeds as an open dense subset of the vortex moduli space,  the conjecture above says that as $e^2 \rightarrow \infty$ the pointwise limit of the vortex metric $\omega_\MM$ over the domain $\mathcal{H}_\xi \subset \MM_\xi$ coincides with the natural $L^2$-metric $\omega_{\mathcal{H}}$. On the complement $\MM_\xi \setminus \mathcal{H}_\xi$ the vortex metric $\omega_\MM$ may very well diverge in this limit. In fact, it is clear from the vortex equation
\[
\Lambda F_A \,- \, i e^2   \Big[  \big( \: \sum_j   Q_j \: |\phi^j|^2  \: \big)  - \tau \Big] \ = \ 0 
\]
that when $e^2$ goes to infinity the curvature $\Lambda F_A$ must explode at the points $x\in M$ where the argument of the square bracket cannot vanish, i.e. at the points such that $\tau$ is not in the interior of the cone $\Delta_{I_{\phi(x)}}$ (see Proposition \ref{embedding}). So $\Lambda F_A$ explodes if the vortex solution is not in the image $\mathcal{H}_\xi \subset \MM_\xi$. It is then plausible that, in the limit $e^2 \rightarrow \infty$, the metric $\omega_\MM$ also diverges outside $\mathcal{H}_\xi$. Moreover, there are known examples where the scalar curvature of $\omega_{\mathcal{H}}$ diverges as one approaches the boundary $\partial \mathcal{H}$.
 So if the conjecture is to hold, at least in these examples the limit of $\omega_\MM$ must also diverge at $\partial \mathcal{H}$. 

Assuming the validity of the conjecture above, one can go one step further by exchanging the limit with the volume integral and propose:

\begin{conjecture}\label{volume_conjecture}
When $\mathcal{H}_\xi$ embeds as an open dense subset of $\MM_\xi$, the volume of $(\mathcal{H}_\xi, \omega_{\mathcal{H}})$ coincides with the limit
\[
\text{\rm Vol}\, (\mathcal{H}_\xi, \omega_{\mathcal{H}}) \ = \ \lim_{e^2 \rightarrow + \infty} \: \text{\rm Vol}\, (\MM_\xi, \omega_{\MM })  \ .
\]
The total scalar curvature of $(\mathcal{H}_\xi, \omega_{\mathcal{H}})$ can also be obtained as the $e^2 \rightarrow  \infty$ limit of the total scalar curvature of $(\MM_\xi, \omega_{\MM})$.
\end{conjecture}
\noindent
This conjecture was used in \cite{Bap2010} to propose explicit formulae for $\text{\rm Vol}\, (\mathcal{H}_\xi, \omega_{\mathcal{H}})$  and the total scalar curvature of $\mathcal{H}$ in the case where $M$ is a Riemann surface and $X$ is a projective space. The volume formula was checked by Speight through a direct computation in the special case of maps $\CC \mathbb{P}^1 \longrightarrow \CC \mathbb{P}^{n-1}$ of degree 1 \cite{Speight}. By the same token, here we can use the results of Section 4 to propose formulae for the volume of the space of maps $M \rightarrow \CC \mathbb{P}^{n-1}$ when $M$ is higher-dimensional.

\subsubsection*{Example}

Take $M$ to be simply connected with $H^2(M, \mathbb{Z}) \simeq  \text{\rm Pic} (M) \simeq \mathbb{Z}$ and the target $X$ to be the projective space $\CC\mathbb{P}^{n-1}$. We take the metric on $\CC\mathbb{P}^{n-1}$ to be $\pi \tau \, \omega_{FS}$, where $\tau > 0 $ and  $\omega_{FS}$ stands for the Fubini-Study form normalized so that its cohomology class generates $H^\ast (\CC\mathbb{P}^{n-1}, \mathbb{Z})$. Any holomorphic map $f: M \rightarrow \CC\mathbb{P}^{n-1}$ has a well defined integer degree $d$. If there exists a positive generator $E \rightarrow M$ of the Picard group and $H \rightarrow  \CC\mathbb{P}^{n-1}$ is the hyperplane bundle, the degree is determined by the equation
\[
f^\ast \, c_1(H) \ = \ d \cdot c_1 (E) \ 
\] 
in $H^2(M, \mathbb{Z})$. Proposition \ref{embedding} says that for large $e^2$ we have an embedding $\mathcal{H}_d \hookrightarrow \MM_d$. Observe that the degree on the vortex moduli space coincides with the degree defined in Example (2) of Section 2.2, as follows from \eqref{pullback_coh}. If the degree $d$ is negative, the moduli space $\MM_d$ is always empty, and hence so is the space of maps $\mathcal{H}_d$. If $d$ is zero, Proposition \ref{toric_moduli_space} says that $\MM_0 \simeq  \CC\mathbb{P}^{n-1}$, hence $\mathcal{H}_0$ is just the space of constant maps $M \rightarrow \CC\mathbb{P}^{n-1}$. If the degree $d$ is positive and $e^2$ is large, Proposition \ref{toric_moduli_space} says that $\MM_d \simeq \CC \mathbb{P}^{\, n r -1}$, where the integer $r$ is the complex dimension of  $H^0(M; E^d)$. 

Now specialize to the case where the base $M = \CC\mathbb{P}^{m}$ is also a projective space. Then according to Proposition \ref{open_embedding_condition} and the remark following it, $\mathcal{H}_d$ will embed as an open dense subset of $\CC \mathbb{P}^{\, n r -1}$ if and only if $n > \dim M$. If this last condition is satisfied we can use the volume computations of Section 4.3, together with Conjecture \ref{volume_conjecture}, to propose the formula
\[
\text{\rm Vol}\, (\mathcal{H}_d, \omega_{\mathcal{H}}) \ = \ \lim_{e^2 \rightarrow + \infty}\, \frac{(\pi \, \sigma)^{n\, r - 1}}{\big(n r -1 \big)!} \  =  \ \frac{(\pi \, \tau \, \text{\rm Vol}\, M)^{n\, r - 1}}{\big(n r -1 \big)!} \ .
\]
Here $r$ is the integer $(m+d)! / (m! \, d!)$, and we are assuming that $m$ and $d$ are positive. The total scalar curvature of $(\mathcal{H}_d, \omega_{\mathcal{H}})$ should be given by 
\[
\int_{\mathcal{H}_d}  s(\omega_{\mathcal{H}}) \ {\rm vol}_\mathcal{H} \ = \  \frac{2 \pi\, nr\,  (nr-1)}{\big[(nr-1)! \big]^{1/(nr-1)}} \ \big( \text{\rm Vol}\, \, \mathcal{H}_d \big)^{(nr-2)/(nr-1)}  \ .
\]

\vspace{.4cm}

\medskip
\noindent
\textbf{Acknowledgements.}\, I would like to thank CAMGSD and Project  PTDC/MAT/119689/2010 of FCT - POPH/FSE for a generous fellowship.

\vspace{.3cm}

\appendix

\section{Appendices}

\subsection{Limits of the vortex equations}

In this appendix we make very informal observations about the strong and weak coupling limits of the vortex equations, i.e. the limits where the constant $e^2$ is large or 
small. We will do so for the simplest abelian equations, namely equations \eqref{vortex1}-\eqref{vortex3} for vortices on a positive line bundle $L \longrightarrow M$.

\subsubsection*{Limit $e^2 \rightarrow \infty$}

In physics this is called the strong coupling limit. Start by noting that for positive $\tau$ the stability condition \eqref{stability1} is satisfied for arbitrarily large values of $e^2$. This means that, as a complex manifold, the vortex moduli space $\MM$ does not change for big and growing values of $e^2$. It always coincides with the space of effective divisors on $\MM$  carrying a homology class Poincar\'e-dual to $c_1(L)$.
Observe from expression \eqref{energy_vortex} that, when $e^2 \rightarrow \infty$, the energy of each vortex solution tends to the finite positive constant
\[
E_{vortex} \ \rightarrow  \ \int_M \frac{2\pi\, \tau}{(m-1)!}\, c_1(L) \ .
\]
On the other hand, looking at definition \eqref{energy_functional} of the functional $E(A, \phi)$, we recognize that for the energy to remain finite in this limit we must have
\[
(\vert \phi \vert^2 \ - \ \tau) \ \rightarrow \ 0
\]
almost everywhere on $M$. This is also suggested by the formal limit $e^2 \rightarrow \infty$ of the second vortex equation \eqref{vortex2}. Note, however, that over the effective divisor $D \subset M$ corresponding a vortex solution $(A, \phi)$, the section $\phi$ is always zero, independently of the value of $e^2$. So for each vortex solution the picture is that  the norm $\vert \phi \vert^2$ tends to $\tau$ everywhere on $M$ except along the corresponding divisor $D$, where $\vert \phi \vert^2$ remains zero. The second vortex equation
\[
i \,\Lambda F_A  \:  + \:  e^2\:  \big( |\phi|^2 - \tau \big) \ = \ 0 
\]
then implies that the curvature $\Lambda F_A$ must explode over $D$ in the strong coupling limit. The ideal picture would be that the curvature gradually concentrates around the divisor as $e^2 \rightarrow \infty$.
Lets now look at the volume of the moduli space in the example where $M$ is an abelian variety. As is explained in Section 2, in this case the moduli space is a projective bundle $\MM \longrightarrow \text{\rm Pic}^0 M$ over the Picard group. The ``size" of the fibres and base of this bundle is determined by expression \eqref{k_class2} for the K\"ahler class. As $e^2 \rightarrow \infty$, we see that the contribution of the fibres remains finite, while the contribution the base becomes smaller and smaller. So, metrically, the moduli space $\MM$ becomes a projective bundle over an infinitesimally small torus.

\subsubsection*{Limit $e^2 \rightarrow 0$ with constant $e^2\, \tau$}

Start by observing that if the stability condition \eqref{stability1} is satisfied for some finite values of $\tau$ and $e^2$, then it is also satisfied for arbitrarily small values of $e^2$ as long as $\tau \, e^2$ remains constant. Once again, this means that the moduli space $\MM$ does not change as a complex manifold when we approach this limit. Now observe that the formal limit of the second vortex equation is in this case
\[
\Lambda F_A  \:  - \:  i\, e^2\, \tau  \ = \ 0 \ . 
\]
This suggests that, for each vortex solution,  when $e^2 \rightarrow 0$ with constant $e^2\, \tau$, the connection $A$ approaches a Hermitian-Einstein connection on $L \longrightarrow M$. If $M$ is simply connected, there is a unique such connection, up to gauge transformations. When $M$ is an abelian variety, the Hermitian-Einstein connections on $L$ are parametrized by $\text{\rm Pic}^0 M$. In this case $\MM$ is a projective bundle over $\text{\rm Pic}^0 M$, and we expect that the connections of the vortex solutions corresponding to the same  projective fibre in $\MM$ will all converge to the same Hermitian-Einstein connection on $L$. Regarding the K\"ahler class of $\MM$, it follows from \eqref{stability1} and  \eqref{k_class2} that both the ``size" of the projective fibres and the ``size" of the base $\text{\rm Pic}^0 M$  diverge in the limit  $e^2 \rightarrow 0$ with constant $e^2\, \tau$.

\subsection{Auxiliary lemmas}

In this appendix we collect two algebraic lemmas about torus actions on vector spaces. These are relevant to understand the stability condition for toric quotients and the related gauged linear sigma-models (GLSM).

\begin{lemma}
Take a constant $\tau \in \Delta \subseteq \mathbb{R}^k$ and assume that condition (H1) of Section 5.2 is satisfied. Then there exists an index subset $I_0 \subseteq \{1, \ldots, n \}$ of cardinality $k$ such that we can write $\tau = \sum_{j \in I_0} \lambda^j \, Q_j$ with positive scalar coefficients. There is no index subset of smaller cardinality with the same property.
\end{lemma}
\begin{proof}
For each non-empty index subset $I \subseteq \{1, \ldots, n \}$ consider the closed cone 
\[
\Delta_I \ : = \ \big\{ v \in \mathbb{R}^k : \  v =   \sum_{j \in I} \lambda^j \, Q_j  \ \text{\rm with } \lambda^j \geq 0 \big\}  \ .
\]
Condition (H1) says that $\tau \not\in \Delta_I$ for any subset $I$ of cardinality less than $k$. Define $I'$ to be an index subset of cardinality $k-1$ such that the euclidean distance between $\tau$ and $\Delta_{I'}$ is minimal among the (positive) distances between $\tau$ and the closed cones corresponding to all index subsets of cardinality $k-1$. Call $v' \in \Delta_{I'}$ the vector inside this cone that minimizes the euclidean distance to $\tau$. Then $\tau - v'$ is a non-zero vector orthogonal to $v'$.
Observe that there exists an element $j_0 \in \{0, \ldots, n\}$ such that the vector $Q_{j_0}$ has positive inner product
\begin{equation} \label{innerproduct}
Q_{j_0} \cdot (\tau - v') \ > \ 0 \ .
\end{equation}
If this were not true, the assumption that $\tau \in \Delta$ would imply that
\[
0 \ \geq \ \tau  \cdot (\tau - v') \ = \ (\tau - v') \cdot (\tau - v') \ ,
\]
which is impossible. Taking this index $j_0$, consider the 1-parameter family of vectors
\[
v_t \ := \ v' \ + \ t\, Q_{j_0} \quad \text{\rm for } t \geq 0 \ .
\]
The euclidean norm of each of these vectors is
\[
\Vert  \tau - v_t \Vert^2 \ = \ \Vert  \tau - v' \Vert^2 \ - \  2\,t\, \,  Q_{j_0} \cdot (\tau - v') \ + \ t^2 \, \Vert  Q_{j_0} \Vert^2 \ ,
\]
so from \eqref{innerproduct} it is clear that for small $\epsilon > 0$ the distance $\Vert  \tau - v_\epsilon \Vert^2$ is strictly smaller than the distance $\Vert  \tau - v' \Vert^2 $. In particular $j_0$ must not belong to $I'$, otherwise $v_\epsilon$ would be inside $\Delta_{I'}$ and this would contradict the definition of $v'$. So the index subset
\[
I_0 \ := \ I' \cup \{ j_0\}
\]
has cardinality $k$. We will now show that $\tau$ belongs to the closed cone $\Delta_{I_0}$. Suppose that this was not true and that $\tau$ was outside $\Delta_{I_0}$. Then the closest point to $\tau$ inside $\Delta_{I_0}$, which we call $v''$, would necessarily be on the boundary $\partial \Delta_{I_0}$. This boundary is the union of the closed cones determined by the subsets of $I_0$ of cardinality $k-1$. In particular there would be a subset $I'' \subset I_0$ of cardinality $k-1$ such that $v'' \in \Delta_{I''}$. Therefore the distance from the cone $\Delta_{I''}$ to $\tau$ would be equal to or smaller than the distance from the cone $\Delta_{I'} \subset \Delta_{I_0}$ to $\tau$. By the minimizing assumptions in the definition of $I'$, this distance cannot be smaller, so it must be equal. Thus
\[
\Vert  \tau - v'' \Vert^2 \ = \ \Vert  \tau - v' \Vert^2
\] 
is the shortest distance between the cone $\Delta_{I_0}$ and $\tau$. This is impossible because above we have found a vector $v_\epsilon \in \Delta_{I_0} $ whose squared distance to $\tau$ is strictly smaller than $\Vert  \tau - v' \Vert^2$. The conclusion is that $\tau$ must belong to $\Delta_{I_0}$, and so be of the form $\tau = \sum_{j \in I_0} \lambda^j \, Q_j$ with $\lambda^j \geq 0$. Finally, assumption (H1) guarantees that none of the coefficients $\lambda^j$ is zero. 
\end{proof}

\begin{lemma}\label{aux_lemma2}
Let  $I \subseteq \{ 1, \ldots, n\}$ be a non-empty index subset, let  $S_I$ denote the subspace of $\mathbb{R}^k$ spanned by the weights $\{Q_j   \in \mathbb{Z}^k : j \in I \}$, and let $\Delta_I \subseteq S_I $ be the closed cone generated by linear combinations of those weights with non-negative coefficients. Then a vector $\sigma \in \mathbb{R}^k$ is in the interior of the cone $\Delta_I$ (regarded as a subset of $S_I$) if and only if it can be written as $\sigma = \sum_{j \in I} \lambda^j\, Q_j$ with strictly positive scalar coefficients.
\end{lemma}
\begin{proof}
The condition that $\sigma$ can be written as $\sum_{j \in I} \lambda^j\, Q_j$ with strictly positive scalar coefficients is clearly an open condition in $S_I$, i.e. if $\sigma$ satisfies it, so will all the vectors in $S_I$ sufficiently close to $\sigma$. This shows that such a $\sigma$ is not on the boundary of the cone $\Delta_I$ (regarded as a subset of $S_I$), and so justifies the ``if" part of the lemma. On the other direction, suppose that the vector $\sigma$ is in the interior of $\Delta_I$. Then for small enough  coefficients $\delta^j > 0$, the vector $\sigma' = \sigma - \sum_{j \in I} \delta^j\, Q_j$ is still in the cone $\Delta_{I}$. In particular we can write $\sigma' = \sum_{j \in I} \lambda'^j\, Q_j$ for some choice of non-negative coefficients $\lambda'^j$ . Thus
\[
\sigma \ = \ \sum_{j \in I} (\lambda'^j + \delta^j ) \, Q_j \ ,
\]
where the coefficients $\lambda'^j + \delta^j$ are all strictly positive, as desired. 
\end{proof}

\begin{lemma}\label{aux_lemma1}
Take a constant $\tau \neq 0$ in $\mathbb{R}^k$ and assume that condition (H1) of Section 5.2 holds. Let  $I \subseteq I'$ be two non-empty index subsets of $\{ 1, \ldots, n\}$. Then, if  we can write $\tau = \sum_{j \in I} \lambda^j \, Q_j$ with positive scalar coefficients, we can also write $\tau = \sum_{j \in I'} \lambda^j \, Q_j$ with positive scalar coefficients. In particular, we can always write $\tau = \sum_{j=1}^n \lambda^j \, Q_j$ with positive scalar coefficients.
\end{lemma}
\begin{proof}
By assumption, the vector $\tau$ is in the interior of the cone $\Delta_{I}$ generated by the weights $\{Q_j \in \mathbb{R}^k : \ j \in I \}$. Condition (H1) implies that this cone is $k$-dimensional. So for small enough positive coefficients $\delta^j$, the vector $\tau' = \tau - \sum_{j \in I' \setminus I} \delta^j\, Q_j$ is still in the interior of $\Delta_{I}$, and hence we can write $\tau' = \sum_{j \in I} \lambda'^j\, Q_j$ for some choice of positive coefficients $\lambda'^j$. Thus
\[
\tau \ = \ \sum_{j \in I} \lambda'^j\, Q_j \ + \ \sum_{j \in I' \setminus I} \delta^j\, Q_j \ ,
\]
as desired. To prove the last assertion, just note that condition (H1) implies that $\tau \in \Delta_I$ for some non-empty index subset of $I \subset  \{ 1, \ldots, n\}$, and hence the assertion follows from the first part of the lemma.
\end{proof}

\vspace{1cm}

\noindent
{\small {\textsc{Centre for Mathematical Analysis, Geometry, and Dynamical Systems (CAMGSD), Instituto Superior T\'ecnico, Av. Rovisco Pais, 1049-001 Lisbon, Portugal} }}

\noindent
{\it Email address:}
{\tt joao.o.baptista@gmail.com}


\end{document}